\documentclass[12pt]{article}
\usepackage[a4paper,margin=1.05in]{geometry}
\usepackage{amsmath,amssymb,amsthm,mathtools,dsfont}
\usepackage{hyperref}
\hypersetup{hidelinks}
\usepackage{enumitem}
\usepackage{microtype}
\usepackage{subcaption}
\usepackage{tikz}
\usetikzlibrary{decorations.pathreplacing}

\def\R{\mathbb{R}}
\def\N{\mathbb{N}}
\newcommand{\ds}{\displaystyle}

\setlist[itemize]{leftmargin=1.5em}
\setlist[enumerate]{leftmargin=1.5em}
\newtheorem{theorem}{Theorem}[section]
\newtheorem{proposition}[theorem]{Proposition}
\newtheorem{lemma}[theorem]{Lemma}
\newtheorem{corollary}[theorem]{Corollary}
\theoremstyle{definition}

\definecolor{aquamarine}{rgb}{0.13, 0.68, 0.8} 

\numberwithin{equation}{section}
\newcommand{\be}{\begin{equation}}
\newcommand{\ee}{\end{equation}}
\newcommand{\baa}{\begin{array}}
\newcommand{\eaa}{\end{array}}
\newcommand{\ba}{\begin{eqnarray}}
\newcommand{\ea}{\end{eqnarray}}

\begin{document}

\title{\bf{Thresholds, fragmentation and symmetrization in parabolic equations}}
\author{Matthieu Alfaro$^{\hbox{\small{ a}}}$, Yihong Du$^{\hbox{\small{ b}}}$, Fran\c cois Hamel$^{\hbox{\small{ c}}}$, and Lionel Roques$^{\hbox{\small{ d }}}$\\
\\
\footnotesize{$^{\hbox{a }}$Univ Rouen Normandie, LMRS, CNRS, Rouen, France}\\
\footnotesize{$^{\hbox{b }}$ School of Science and Technology, University of New England, Armidale, NSW 2351, Australia}\\
\footnotesize{$^{\hbox{c }}$Aix Marseille Univ, CNRS, I2M, Marseille, France} \\
\footnotesize{$^{\hbox{d }}$INRAE, BioSP, 84914, Avignon, France}}
\date{}

\maketitle

\begin{abstract}
\noindent{This paper is mainly concerned with the large-time dynamics of bounded nonnegative solutions of reaction-diffusion equations on the real line with  nonlinearities mainly of the bistable-type.  We first consider initial data of the type $\alpha\,\mathds{1}_I$, namely scalar multiples of indicator functions of intervals $I$. For each amplitude $\alpha$, the existence of a threshold length $L^*(\alpha)$ separating the extinction and the persistence of the solutions at large time is known. We here address the question as to whether the limit of the threshold sizes $L^*(\alpha)$ as $\alpha\to+\infty$ is positive or zero. We provide sufficient conditions under which this limit is positive, and others under which it is zero. Secondly, when the limit is positive, we show that some fragmented initial data, which are equally distributed as $\alpha\,\mathds{1}_I$, give rise to solutions persisting at large time, whereas the solutions emanating from $\alpha\,\mathds{1}_I$ go to extinction at large time. This result shows an unexpected favourable effect of the fragmentation of the initial datum on the large-time dynamics. Last, we consider the mass concentration principle for parabolic equations, which states that nonnegative solutions can be controlled from above in an integral sense by the solutions emanating from the Schwarz symmetrically decreasing rearrangements of the initial data. Since the pioneering results of [Bandle, 1976] and [Alvino, Trombetti, Lions, 1990], this principle is known to hold in bounded domains with Dirichlet boundary conditions under different types of assumptions of the coefficients of the equation. We show that this mass concentration principle is not valid in general, even for simple equations of the type $\partial_t u = \partial_{xx} u + f(u)$.}
\end{abstract}

\tableofcontents


\section{Introduction and main results}\label{sec1}

This paper is concerned with some questions related to the large-time behaviour of solutions of the one-dimensional reaction-diffusion equation
\begin{equation}\label{eq:PDE}
  \begin{cases}
    \partial_t u = \partial_{xx} u + f(u), & t>0,\ x\in\R,\\[2pt]
    u(0,x) = u_0(x), & x\in\R,
  \end{cases}
\end{equation}
where the initial datum $u_0\in L^\infty(\R)$ is nonnegative. Throughout the paper, we assume that the nonlinearity $f:\R_+\to\R$ is locally Lipschitz continuous and differentiable at $0$. Furthermore, in most results, $f$ is bistable with positive mass over $[0,1]$, namely there exists $\theta\in(0,1)$ such that
\begin{equation}\label{hyp:f1a}
\left\{
\baa{ll}
f(0)=f(\theta)=f(1)=0, & f'(0)<0,\vspace{3pt}\\
f(s)<0\ \text{ for } s\in(0,\theta), & f(s)>0\ \text{ for } s\in(\theta,1),\vspace{3pt}\\
\ds\int_0^1 f(s)\,ds>0,&
\eaa
\right.
\end{equation}
and
\begin{equation}\label{hyp:f1b}
f(s)<0\ \text{ for } s>1.
\end{equation}
The solution $u=u(t,x)$ of~\eqref{eq:PDE} is then the unique bounded nonnegative function defined in~${(0,+\infty)\times\R}$ such that $u(t,\cdot)\to u_0$ as $t\to0^+$ in $L^1_{loc}(\R)$. Furthermore, from standard parabolic regularity and the maximum principle, $u$ is a classical solution in~$(0,+\infty)\times\R$, $u(t,x)>0$ for all $t>0$ and $x\in\R$ provided that $\|u_0\|_{L^\infty(\R)}>0$, and $\|u(t,\cdot)\|_{L^\infty(\R)}\le\max(\|u_0\|_{L^\infty(\R)},1)$ for all $t>0$. 

In this paper, we address the following question: how does the fate of the solution~$u$ depend on its initial datum $u_0$? More precisely, we are interested in the threshold between the convergence of $u(t,\cdot)$ to $0$ as~$t\to+\infty$ and the convergence to a positive steady state, in terms of the amplitude of $u_0$ and of the size and fragmentation of its support.

We first consider compactly supported initial data of the form
\begin{equation}\label{eq:IC-alpha-L}
u_0(x)=u_\alpha^L(0,x):=\alpha\,\mathds{1}_{(-L/2,L/2)}(x)=\left\{\baa{ll}\!\alpha & \hbox{if $x\in(-L/2,L/2)$},\vspace{3pt}\\ \!0 & \hbox{if $x\in\R\setminus(-L/2,L/2)$},\eaa\right.
\ee
for $\alpha>0$ and $L>0$, and we write $u_\alpha^L$ the corresponding solution of \eqref{eq:PDE}. It follows from~\eqref{hyp:f1a} and the maximum principle that, if $\alpha\in(0,\theta]$, then $u^L_\alpha(t,\cdot)\to0$ as $t\to+\infty$ uniformly in $\R$. The sharp-threshold result of Zlato\v{s}~\cite{Zla06} implies that, for every $\alpha\in(\theta,1]$, there exists a critical length $L^*(\alpha)\in(0,+\infty)$ separating extinction and propagation.

\begin{theorem}[Zlato\v{s} \cite{Zla06}]\label{thm:Zla}
Assume that $u_0$ is given by \eqref{eq:IC-alpha-L} and that $f$ satisfies~\eqref{hyp:f1a}. For every $\alpha\in(\theta,1]$, there exists $L^*(\alpha)\in(0,+\infty)$ such that
\be\label{trichotomy}
\!\lim_{t\to+\infty}\!u_\alpha^L(t,\cdot)\!=\!\left\{
\begin{array}{llll}
\!\!0 & \!\!\text{uniformly in } \R & \!\!\text{if } 0\!<\!L\!<\!L^*(\alpha) & \!\!\!(\!\hbox{``extinction''}),\\[1mm]
\!\!p & \!\!\text{uniformly in } \R & \!\!\text{if } L=L^*(\alpha) & \!\!\!(\!\hbox{``cv. to ground state''}),\\[1mm]
\!\!1 & \!\!\text{locally uniformly in } \R & \!\!\text{if } L\!>\!L^*(\alpha) & \!\!\!(\!\hbox{``propagation''}),
\end{array}\right.
\ee
where $p$ is the positive symmetrically decreasing solution of
\begin{equation}\label{eq:U}
p''+f(p)=0\hbox{ in $\R$},\ \ p(0)=\theta^*,\ \ p'(0)=0,\ \ p'(x)<0\hbox{ for $x>0$},\ \ p(\pm\infty)=0,
\end{equation}
with $\theta^*\in(\theta,1)$ determined by
\[
\int_0^{\theta^*}\!\!f(s)\,ds=0.
\]
\end{theorem}

The ground state $p$ therefore acts as a transition between extinction and propagation regimes. For general results on the existence of such ground state solutions $p$ for elliptic equations in any spatial dimension, we refer e.g. to the seminal paper~\cite{BerLio83}.

Theorem~\ref{thm:Zla} has been extended in \cite{DuMat10} to general increasing families of compactly supported initial data.  Let
\[
X_+ := \bigl\{ \phi \in L^\infty(\R) : \phi \ge 0,\ \phi\text{ is compactly supported}\bigr\}.
\]
Consider a family $\{\phi_\lambda\}_{\lambda\ge 0} \subset X_+$ satisfying
\begin{equation} \label{hyp:DuMat}
\left\{
\begin{aligned}
   &\phi_\lambda \in X_+ \text{ for all } \lambda\ge0,\ \text{and } 
     \lambda \mapsto \phi_\lambda \text{ is continuous from } \R_+ \text{ to } L^1(\R), \\
    &0 \le \lambda_1 < \lambda_2 \ \Longrightarrow\
     \phi_{\lambda_1} \le \phi_{\lambda_2}\hbox{ in $\R$}\ \hbox{ and }\ 
     \|\phi_{\lambda_2}-\phi_{\lambda_1}\|_{L^\infty(\R)}>0,\\
    & \phi_0=  0 \text{ a.e. in }\R.
\end{aligned}
\right.
\end{equation}

\begin{theorem}[Du, Matano \cite{DuMat10}]\label{thm:DM}
Assume that $f$ satisfies \eqref{hyp:f1a}-\eqref{hyp:f1b}. Let $\{\phi_\lambda\}_{\lambda\ge0}\subset X_+$ satisfy~\eqref{hyp:DuMat} and, for each $\lambda\ge0$, let $u_\lambda$ denote the solution of \eqref{eq:PDE} with initial datum $\phi_\lambda$.  Then there exist $\lambda_* \in (0,+\infty]$ and $x_0 \in \R$ such that
\[
\lim_{t\to+\infty} u_\lambda(t,\cdot) = \left\{
\begin{array}{lll}
0 & \text{uniformly in } \R & \text{if }  0 \le \lambda < \lambda_*, \\[1mm]
p(\cdot-x_0) & \text{uniformly in } \R & \text{if }  \lambda = \lambda_*, \\[1mm]
1 & \text{locally uniformly in } \R & \text{if }  \lambda > \lambda_*,
\end{array}\right.
\]
where $p$ is as in Theorem~$\ref{thm:Zla}$.
\end{theorem}

The first results on extinction for ``small" initial data and propagation for ``large" initial data were established in~\cite{AroWei78,FifMcL77,Kan-64} for a larger class of functions~$f$. In addition to the aforementioned papers~\cite{DuMat10,Zla06}, the existence of sharp threshold between extinction and propagation (among monotone families of initial data) was investigated more recently, through different techniques. We refer to~\cite{MatPol16} for non-compactly-supported initial data converging to~$0$ at $\pm\infty$,~\cite{Mur-Zho-13,Mur-Zho-17} for non-compactly-supported radially symmetric non-increasing initial data in $\R$ or $\R^N$. The reference~\cite{LewKar93} provides some enlightening numerical simulations and  a formula of the threshold radius for indicator functions supported on disks and giving rise to solutions converging to a ground state, for the equivalent of~\eqref{eq:PDE} in~$\R^2$ with large cubic nonlinearities~$f$. For initial data~\eqref{eq:IC-alpha-L}, quantitative estimates on~$L^*(\alpha)$ as~$\alpha\to\theta^+$ have been obtained in~\cite{Alf-Duc-Fay-20}. Existence, uniqueness or non-uniqueness of thresholds have also been established for other equations than~\eqref{eq:PDE}, such as equations with super-linearly growing reactions $f$~\cite{FeiPet97}, equations with time-dependent reactions~$f(t,u)$ satisfying $\sup_{t\ge0}\frac{\partial f}{\partial u}(t,0)<0$~\cite{Pol11}, or equations with non-local diffusion~\cite{AlfDuc23,ZhaLi22} or non-local reactions~\cite{AlfChu26}.

Coming back to~\eqref{eq:PDE} with $f$ satisfying \eqref{hyp:f1a}-\eqref{hyp:f1b}, for each $\alpha>0$, the family $\{u^L_\alpha(0,\cdot)\}_{L\ge0}$ given in~\eqref{eq:IC-alpha-L} (with $u^0_\alpha(0,\cdot):=0$) satisfies the condition of Theorem~\ref{thm:DM}. Furthermore, the solution $u^L_\alpha$ with initial datum $u^L_\alpha(0,\cdot)$ is even in $x$. Therefore, combining Theorems~\ref{thm:Zla} and~\ref{thm:DM} together with the parabolic comparison principle, it follows that, under conditions~\eqref{hyp:f1a}-\eqref{hyp:f1b}, for each $\alpha\in(\theta,+\infty)$, there is $L^*(\alpha)\in(0,+\infty)$ such that~\eqref{trichotomy} holds. Moreover, the map $\alpha\mapsto L^*(\alpha)$ is immediately non-increasing, and even decreasing, in $(\theta,+\infty)$ (see Lemma~\ref{lem:alpha-monot} below for more details). We can therefore define
\[
L_\infty:=\lim_{\alpha\to+\infty}L^*(\alpha)\in[0,+\infty).
\]

The parabolic maximum principle yields the following corollary (see below for the proof).

\begin{corollary}\label{cor:limitL}
Assume that $f$ satisfies \eqref{hyp:f1a}-\eqref{hyp:f1b}. The following properties hold.
\begin{enumerate}
    \item[(i)] If $L_\infty>0$ and $0<L \le L_\infty$,  then, for every $\alpha>0$,
    \[
        \lim_{t\to+\infty}u_\alpha^L(t,\cdot)=0\quad \text{uniformly in } \R.
    \]
    \item[(ii)] If $L>L_\infty$, then there exists $\alpha_0>\theta$ such that, for every $\alpha>\alpha_0$,
    \[
        \lim_{t\to+\infty}u_\alpha^L(t,\cdot)=1\quad \text{locally uniformly in } \R.
    \]
\end{enumerate}
\end{corollary}

The first objective of this manuscript is to understand under which conditions on $f$ the real number $L_\infty$ is positive. The second objective is to use this information to derive unexpected results on favourable effects of fragmentation of the initial datum. In particular, we address the following question: among initial data with the same distribution function as $\alpha\,\mathds{1}_{(-L/2,L/2)}$ (see later for the precise definition), is the interval of size $L$  always the most favourable shape for propagation?  The third objective is to relate our results to the mass concentration principle, which claims that nonnegative solutions of semilinear parabolic equations can be controlled from above in an integral sense by the solutions emanating from the Schwarz symmetrically decreasing rearrangements of the initial data. As a matter of fact, we show with two different counter-examples, the first ones in the literature to our knowledge for problems like~\eqref{eq:PDE}, that this mass concentration principle is not valid in general.


\subsection{Large-amplitude asymptotics}

The next two theorems describe two opposite possible behaviours of the threshold lengths~$L^*(\alpha)$ as $\alpha\to+\infty$. The first one  gives sufficient conditions on~$f$ ensuring that $L_\infty$ is positive. In other words, for such $f$, even very large amplitudes lead to extinction if the support is too narrow.

\begin{theorem}\label{thm:limitL}
Assume that $f$ satisfies \eqref{hyp:f1a} and that there exist $p>2$ and $A>0$ such that
\begin{equation}\label{hyp:f2}
   f(s)\le -A\,(s-1)^p \quad \text{for all } s> 1.
\end{equation}
\begin{enumerate}
    \item[(i)] If $p>3$, then $L_\infty>0$ for every $A>0$.
    \item[(ii)] If $2<p\leq 3$, then there exists $A^*>0$, depending only on $f|_{[0,1]}$ and on $p$, such that $L_\infty>0$ whenever $A>A^*$.
\end{enumerate}
\end{theorem}

The second theorem deals with the opposite situation,  giving sufficient conditions on~$f$ ensuring that $L_\infty=0$. In other words, for such $f$, for every $L>0$, even an arbitrarily small support can lead to propagation, provided the amplitude $\alpha$ is large enough.

\begin{theorem}\label{th:Linfty-psmall}
Assume that $f$ satisfies \eqref{hyp:f1a}, and that there exist $0\leq p<3$ and $a>0$ such that
\begin{equation}\label{hyp:f3}
    0>f(s)\ge -a\,s^p
    \quad \text{for all } s>1.
\end{equation}
Then the following properties hold.
\begin{enumerate}
    \item[(i)] If $0\leq p \leq 1$, then
    \[
    0<\liminf_{\alpha\to+\infty}\alpha\,L^*(\alpha)
    \le
    \limsup_{\alpha\to+\infty}\alpha\,L^*(\alpha)<+\infty.
    \]
    In particular, $L_\infty=0$.
    \item[(ii)] If $1<p<3$, then there exists $a_*>0$, depending only on $f|_{[0,1]}$ and on $p$, such that $L_\infty=0$ whenever $0<a<a_*$.
\end{enumerate}
\end{theorem}

The shape of the nonlinearity $f$ on $(1,+\infty)$, in particular its decay at infinity, is obviously of crucial importance to determine whether $L_\infty>0$ or $L_\infty=0$. However, the reader will have noticed that, in the intermediate range $2<p<3$, the two behaviours may occur: strong enough absorption ($A>A^*$) gives $L_\infty>0$, whereas sufficiently weak absorption ($0<a<a_*$)  gives $L_\infty=0$. Thus the exponent $p$ alone does not determine the large-amplitude limit: the shape of $f$ {\it on the whole of} $(1,+\infty)$ is relevant (and not only the behaviour at infinity).

The possibility that $L_\infty$ be positive or zero is related to the competition between the  relaxation of large amplitudes and the amount of mass that diffusion can transfer outside the initial support before this relaxation has occurred. Indeed, for $p>0$ and $A>0$, consider the solution of the ordinary differential equation
\be\label{eq:ode}\left\{\baa{l}
U'(t)=-A\,(U(t)-1)^p,\quad t>0,\vspace{3pt}\\
U(0)=\alpha>1.\eaa\right.
\ee
The solution $U$ is defined and decreasing in $[0,T)$, with $T=+\infty$ if $p\ge1$, and $T=\frac{(\alpha-1)^{1-p}}{A(1-p)}$ if $p<1$. Furthermore, $U(t)\to1^+$ as $t\nearrow T$. For every fixed $M>1$ and for any $\alpha>M$, the time needed for $U$ to decrease from $\alpha$ to $M$ is
\[
T_{\alpha\to M}
=\frac1A\int_M^\alpha \frac{ds}{(s-1)^p}.
\]
Hence $T_{\alpha\to M}$ remains bounded as $\alpha\to+\infty$ if and only if $p>1$. 

When $0<p\le1$, this relaxation time diverges as~$\alpha\to+\infty$, leaving diffusion enough time to spread a large amount of mass outside the initial support $[-L/2,L/2]$, even if $L>0$ is small. This is the mechanism behind Theorem~\ref{th:Linfty-psmall}-(i). Actually, this heuristic for $0<p\le1$ holds when $f$ satisfies~\eqref{hyp:f1a} and~\eqref{hyp:f3}, whereas in~\eqref{hyp:f2} we always assume $p\ge1$ to guarantee the local Lipschitz continuity of $f$ around $1$ (since $f(1)=0$).

When $p>1$, the relaxation time from height $\alpha\gg1$ down to values of order $1$ is uniformly bounded with respect to $\alpha$. Under assumptions~\eqref{hyp:f1a} and~\eqref{hyp:f2}, the proof of Theorem~\ref{thm:limitL} makes this intuition quantitative through the $L^1$ estimate
\be\label{ineqB}
\forall\,t>0,\quad\|u_\alpha^L(t,\cdot)\|_{L^1(\R)}\le L\,B(t)+C\,e^{\gamma t}\int_0^t \frac{B(s)}{\sqrt{t-s}}\,ds,
\ee
where $C>0$ and $\gamma>0$ are independent of $\alpha>1$ and of $L>0$, and
\[
B(s)=1+(A\,(p-1)\,s)^{-\frac1{p-1}}
\]
is the ODE (ordinary differential equation) upper bound of $\|u_\alpha^L(t,\cdot)\|_{L^\infty(\R)}$, see~\eqref{eq:L1-bound} below. The integral in~\eqref{ineqB} is understood as $+\infty$ when $p\le2$, in which case~\eqref{ineqB} does not provide any information. In~\eqref{ineqB}, the first term measures the mass that remains inside the initial interval, while the second one controls the mass transferred by diffusion to the exterior region through the edges of the support. The behaviour of the latter term near~$t=0$ leads to two regimes where the above $L^1$ strategy yields $L_\infty>0$:
\begin{enumerate}
\item[-] {\it The case $p>3$.} In this case,
\[
\int_0^t \frac{B(s)}{\sqrt{t-s}}\,ds\to0\quad \text{as } t\to0^+.
\]
Hence one can first choose a short time $t_1>0$ so that the mass diffused outside the initial interval is uniformly small with respect to $\alpha$, and then choose $L>0$ small enough so that the full $L^1$ norm falls below an extinction threshold. This gives $L_\infty>0$ for every $A>0$.
\item[-] {\it The case $2<p\leq 3$.} The preceding short-time argument no longer applies. However, the singularity of $B(s)/\sqrt{t-s}$ at $s=0$ is still integrable when $p>2$. The estimate~\eqref{ineqB} can therefore be evaluated at a suitable positive time, and the resulting bound becomes small when $A$ is large enough. This yields $L_\infty>0$ for $p\in(2,3]$, provided the damping coefficient $A$ is sufficiently large.
\end{enumerate}

The mechanisms leading to $L_\infty=0$ are different:
\begin{enumerate}
\item[-] {\it The case where the negative reaction above $1$ is at most linear.} If $0\leq p\leq 1$ and~\eqref{hyp:f3} holds, then there exists $\rho>0$ such that
\[
f(s)\ge -\rho s \quad \text{for all } s\ge0.
\]
Thus large amplitudes cannot be damped faster than linearly. Very large initial amplitudes can then generate, after a fixed positive time, a supercritical plateau on an interval of prescribed length. This yields $L_\infty=0$, and more precisely $L^*(\alpha)$ is of order $1/\alpha$ as $\alpha\to+\infty$.
\item[-] {\it The case $1<p<3$ with weak absorption.} This is the second mechanism leading to $L_\infty=0$. Here the ODE relaxation time is already uniformly bounded with respect to $\alpha$, so the previous linear comparison argument no longer applies. Instead, the proof of Theorem~\ref{th:Linfty-psmall}-(ii) uses a self-similar profile of the absorption equation
\[
u_t=u_{xx}-u^p.
\]
Such a profile exists in dimension one precisely for $1<p<3$. After a suitable rescaling, it gives a subsolution which lies below a large enough multiple of the indicator function of any prescribed interval and which, at a later time, is larger than $1$ on an interval of length $R>L^*(1)$. The comparison principle then implies propagation.
\end{enumerate}

Although these results cover a broad range of regimes, two natural questions remain open:
\begin{enumerate}
\item[-] {\it The case $1<p\leq 2$.} We know from Theorem~\ref{th:Linfty-psmall}-(ii) that $L_\infty=0$ when the coefficient $a$ in~\eqref{hyp:f3} is small enough. It remains unclear whether this should hold for every $a>0$. 
\item[-] {\it The borderline case $p=3$.} Theorem~\ref{thm:limitL} gives $L_\infty>0$ when the coefficient $A$ in~\eqref{hyp:f2} is large enough. Conversely, the self-similar construction used in Theorem~\ref{th:Linfty-psmall}-(ii) stops at $p<3$. The weak-absorption regime at $p=3$, for instance under a lower bound of the form~\eqref{hyp:f3} with small coefficient $a$, is therefore not settled by the present arguments.
\end{enumerate}


\subsection{A favourable effect of fragmentation}

The other main results of the paper are related to the analysis of the effect of fragmentation of the support of the initial datum, which is now of a more general type than~\eqref{eq:IC-alpha-L}, but still comparable to it in some sense. Our results show that a sufficiently fine fragmentation of the initial support may turn extinction into propagation.  Let us recall that two nonnegative measurable functions $u_0$ and $v_0$ that converge to $0$ at $\pm\infty$
are said to have the same distribution function, or to be equally distributed,
if
\[
|\{x:u_0(x)>\lambda\}|=|\{x:v_0(x)>\lambda\}|
\qquad\text{for every } \lambda\ge0,
\]
where $|E|$ stands for the Lebesgue measure of a measurable set $E\subset\R$. In particular, if $\alpha>0$ and $E,F$ are measurable  bounded sets, then
$\alpha\mathds{1}_E$ and $\alpha\mathds{1}_F$ are equally distributed
if and only if $|E|=|F|$.

For $0<R_1<R_2$ and $n\in\N\setminus\{0\}$, define (see Fig.~\ref{fig:DI}):
\begin{equation}\label{eq:En}
E_n(R_1,R_2):=
\bigcup_{k=0}^{n-1}
\left(-\frac{R_2}{2}+\frac{kR_2}{n},\ -\frac{R_2}{2}+\frac{kR_2}{n}+\frac{R_1}{n}\right)\!\!.
\end{equation}
Thus the Lebesgue measure $|E_n(R_1,R_2)|$ of $E_n(R_1,R_2)$ is equal to $R_1$, but this set is spread across an interval of length approximately equal to $R_2$ when $n$ is large.

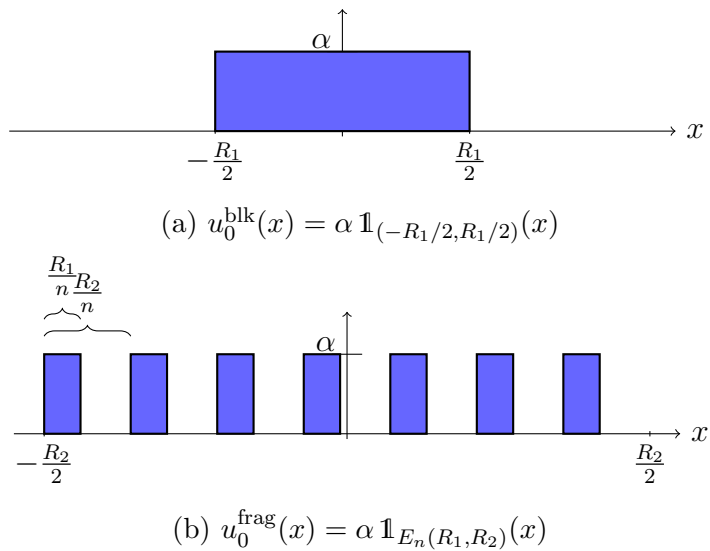
\begin{figure}[h!]
\centering
\begin{subfigure}{0.9\textwidth}
\centering
\begin{tikzpicture}[x=8cm,y=1.4cm]
\def\Rone{0.42}   
\def\alp{0.75}    
\draw[->] (-0.55,0) -- (0.55,0) node[right] {$x$};
\draw[->] (0,-0.05) -- (0,1.15);
\fill[blue!60] (-\Rone/2,0) rectangle (\Rone/2,\alp);
\draw[thick] (-\Rone/2,0) rectangle (\Rone/2,\alp);
\draw (-\Rone/2,-0.025) -- (-\Rone/2,0.025);
\draw (\Rone/2,-0.025) -- (\Rone/2,0.025);
\node[below] at (-\Rone/2,0) {$-\frac{R_1}{2}$};
\node[below] at (\Rone/2,0) {$\frac{R_1}{2}$};
\draw (-0.025,\alp) -- (0.025,\alp);
\node[left] at (0,\alp+0.1) {$\alpha$};
\end{tikzpicture}
\caption{$u_0^{\rm blk}(x)=\alpha\,\mathds{1}_{(-R_1/2,R_1/2)}(x)$}
\end{subfigure}
\vspace{1em}
\begin{subfigure}{0.9\textwidth}
\centering
\begin{tikzpicture}[x=8cm,y=1.4cm]
\def\Rone{0.42}   
\def\n{7}         
\def\alp{0.75}    
\draw[->] (-0.55,0) -- (0.55,0) node[right] {$x$};
\draw[->] (0,-0.05) -- (0,1.15);
\foreach \k in {0,...,6} {
    \pgfmathsetmacro{\a}{-0.5 + \k/\n}
    \pgfmathsetmacro{\b}{-0.5 + (\k+\Rone)/\n}
    \fill[blue!60] (\a,0) rectangle (\b,\alp);
    \draw[thick] (\a,0) rectangle (\b,\alp);
}
\draw (-0.5,-0.025) -- (-0.5,0.025);
\draw (0.5,-0.025) -- (0.5,0.025);
\node[below] at (-0.5,0) {$-\frac{R_2}{2}$};
\node[below] at (0.5,0) {$\frac{R_2}{2}$};
\draw (-0.025,\alp) -- (0.025,\alp);
\node[left] at (0,\alp+0.1) {$\alpha$};
\draw[decorate,decoration={brace,amplitude=4pt}]
    (-0.5,0.92) -- (-0.5 + 1/\n,0.92)
    node[midway,above=5pt] {$\frac{R_2}{n}$};
\draw[decorate,decoration={brace,amplitude=4pt}]
    (-0.5,1.08) -- (-0.5 + \Rone/\n,1.08)
    node[midway,above=5pt] {$\frac{R_1}{n}$};
\end{tikzpicture}
\caption{$u_0^{\rm frag}(x)=\alpha\,\mathds{1}_{E_n(R_1,R_2)}(x)$}
\end{subfigure}
\caption{Schematic representation of the initial data $u_0^{\rm blk}$ and $u_0^{\rm frag}$ in Theorem~\ref{thm:frag-fav}. The two data have the same distribution function and their supports have measure $R_1$. For suitable choices of $R_1$, $R_2$, $\alpha$, and $n$, Theorem~\ref{thm:frag-fav} shows that $u_0^{\rm blk}$ leads to extinction, whereas $u_0^{\rm frag}$ leads to propagation.}
\label{fig:DI}
\end{figure}

\begin{theorem}\label{thm:frag-fav}
Assume that $f$ satisfies~\eqref{hyp:f1a}-\eqref{hyp:f1b} and $L_\infty>0$. Then, for any real numbers~$R_1$ and $R_2$ such that $0<R_1<L_\infty<R_2$, there exist $\alpha>\theta$ and an integer $n\ge2$ such that, for the two bang-bang data
\[
    u_0^{\rm blk}:=\alpha\,\mathds{1}_{(-R_1/2,R_1/2)},
    \qquad
    u_0^{\rm frag}:=\alpha\,\mathds{1}_{E_n(R_1,R_2)},
\]
which have the same distribution function, the solution $u^{\rm blk}$ of~\eqref{eq:PDE} with initial datum~$u_0^{\rm blk}$ extinguishes and the solution $u^{\rm frag}$ of~\eqref{eq:PDE} with initial datum $u_0^{\rm frag}$ propagates, that is,~$u^{\rm blk}(t,\cdot)\to0$ as $t\to+\infty$ uniformly in $\R$ and $u^{\rm frag}(t,\cdot)\to1$ as $t\to+\infty$ locally uniformly in~$\R$.
\end{theorem}

Theorem~\ref{thm:frag-fav} shows that a more fragmented initial datum can be more likely to persist at large time. The proof of this counter-intuitive result is based on the fact that diffusion rapidly merges the fragments of the set $E_n(R_1,R_2)$ and creates a broad supercritical region for propagation to occur.

The effect of the spatial distribution of initial data on the large-time dynamics of solutions of equations such as~\eqref{eq:PDE} has been the subject of several earlier references. In~\cite{GarRoqHam12}, for~\eqref{eq:PDE} under assumptions~\eqref{hyp:f1a}, initial data $\mathds{1}_{(-d/2-L/2,-d/2)}+\mathds{1}_{(d/2,d/2+L/2)}$ (with $d>0$ and $L>0$) have been considered and the critical threshold length $L^*_d>0$ separating extinction (for $L<L^*_d$) and propagation (for $L>L^*_d$) has been shown to converge to $2\,L^*(1)$ as $d\to+\infty$. Numerical simulations in~\cite{GarRoqHam12} suggest that $L^*_d$ may be  decreasing with respect to $d$ for small $d>0$, and~\cite{Nad-26} gives a formula for the derivative of $L^*_d$ with respect to $d$, in terms of the unique positive solution of the adjoint parabolic equation linearized around the threshold solution with $L=L^*_d$. For higher-dimensional versions of~\eqref{eq:PDE}, various fragmentation indices of initial distributions of the type $u_0=\mathds{1}_E$ have been used in~\cite{AlfHamRoq24}, defined in terms of the suitably renormalized distance from $E$ to the set of equimeasurable balls, either in the $L^1$ sense or in the Hausdorff sense. These indices are invariant by rigid motion or scaling of $E$. It follows from~\cite{AlfHamRoq24} that in general there is no monotonicity of the large-time dynamics with respect to the fragmentation indices, among equimeasurable sets $E$. These results however concern pairs of sets which are truly fragmented, far from Euclidean balls. We also mention that, for reaction-diffusion equations in bounded domains $\Omega$ with Neumann boundary conditions, the papers~\cite{Hal-Sma-22,Maz-Nad-Tol-21,Nad-Tol-20} have dealt with the optimization of integral quantities such as $\int_\Omega u(T,\cdot)$ with respect to the initial datum under some pointwise and integral constraints.

Coming back to~\eqref{eq:PDE}, Theorem~\ref{thm:frag-fav} has several consequences, which shed light on the unexpected role of fragmentation of the initial support on the large-time dynamics of solutions of~\eqref{eq:PDE}. In the sequel, for real numbers $a_1<b_1<a_2<b_2<\cdots<a_k<b_k$, with $k\in\N\setminus\{0\}$, let us denote
$$E_{a_1,b_1,\cdots,a_k,b_k}:=(a_1,b_1)\cup(a_2,b_2)\cup\cdots\cup(a_k,b_k),$$
whose Lebesgue measure is equal to $|E|=(b_1-a_1)+\cdots+(b_k-a_k)$.

\begin{corollary}\label{cor:one-hole}
Assume that $f$ satisfies~\eqref{hyp:f1a}-\eqref{hyp:f1b} and $L_\infty>0$. Then there exist $\alpha>\theta$, an integer $k\in\N\setminus\{0\}$, and some real numbers
$$a_1<b_1<a_2<b_2<\cdots<a_{k-1}<b_{k-1}<a_k<b_k\ \hbox{ and }\ a_k<b'_k<a'_{k+1}<b'_{k+1}$$
such that the sets
$$\mathcal{E}:=E_{a_1,b_1,\cdots,a_k,b_k}\ \hbox{ and }\ \mathcal{F}:=E_{a_1,b_1,\cdots,a_k,b'_k,a'_{k+1},b'_{k+1}}$$
have the same Lebesgue measure $($in other words, $\mathcal{F}$ is obtained from $\mathcal{E}$ by splitting the connected component $(a_k,b_k)$ of $\mathcal{E}$ into two disjoint components $(a_k,b'_k)\cup(a'_{k+1},b'_{k+1})$ of the same total Lebesgue measure$)$ and the solution of~\eqref{eq:PDE} with initial datum $\alpha\,\mathds{1}_{\mathcal{E}}$ extinguishes, while the solution of~\eqref{eq:PDE} with initial datum $\alpha\,\mathds{1}_{\mathcal{F}}$ does not. 
\end{corollary}

The case $k=1$ in Corollary~\ref{cor:one-hole} would mean that the solution of~\eqref{eq:PDE} emanating from a multiple $\alpha$ of the indicator function of a single interval extinguishes while the solution emanating from an equally distributed initial datum supported on two disjoint intervals does not. Although the exact value of $k$ is not known, Corollary~\ref{cor:one-hole} brings us closer to the conjecture formulated in~\cite{GarRoqHam12,Nad-26} (for $\alpha=1$), which compares the solutions emanating from the indicator function of an interval with that of the union of two intervals of half the length.


\subsection{Two counter-examples to the mass concentration principle}

In another perspective, much work has been devoted to the analysis of the effect of rearrangements on nonnegative solutions of parabolic equations. In particular, Bandle~\cite{Ban76} and Alvino, Trombetti and Lions~\cite{AlvLioTro90} obtained mass concentration comparison results for parabolic equations, under certain assumptions, in bounded domains with homogeneous Dirichlet boundary conditions. Roughly speaking, these results say that any solution with a nonnegative initial datum~$u_0$ is dominated, in an integral sense, by the solution of the same equation with initial datum equal to the radially symmetric non-increasing rearrangement of $u_0$. More precisely, let
$$I_R:=(-R,R),$$
with $R\in(0,+\infty]$. When $R=+\infty$, we understand $I_R=\R$.  
In this section, we denote by
\begin{equation}\label{eq:PDE-R}
  \begin{cases}
    \partial_t u = \partial_{xx}u+f(u), & t>0,\ x\in I_R,\\[2pt]
    u(0,x)=u_0(x), & x\in I_R,
  \end{cases}
\end{equation}
the analogue of~\eqref{eq:PDE} in the  interval $I_R$.

For a nonnegative function $w\in L^\infty(I_R)$  (assuming in addition that it converges to $0$ at $\pm\infty$ if $R=+\infty$), we denote by $w^*$ its Schwarz non-increasing rearrangement, that is, $w^*:I_R\to\R_+$ is nonnegative, even, non-increasing with respect to~$|x|$,  and it has the same distribution function as~$w$, namely 
\be\label{defdistri}
\forall\,\lambda\ge0,\quad
|\{x\in I_R:w(x)>\lambda\}|
=
|\{x\in I_R:w^*(x)>\lambda\}|.
\ee
The function $w^*$ can be made unique by imposing that, say, it be right-continuous with respect to $|x|$. For a locally Lipschitz-continuous function $f:\R_+\to\R$ with $f(0)=0$, consider   problem~\eqref{eq:PDE-R}   with a nontrivial nonnegative initial datum $u_0\in L^\infty(I_R)$, and the same problem with the rearranged initial datum, namely
\begin{equation}\label{eq:PDE2}
  \begin{cases}
    \partial_t v = \partial_{xx}v+f(v), & t>0,\ x\in I_R,\\[2pt]
    v(0,x)=u_0^*(x), & x\in I_R.
  \end{cases}
\end{equation}
If $R<+\infty$, both $u$ and $v$ are supplemented with homogeneous Dirichlet boundary conditions
\be\label{eq:dir}
u(t,\pm R)=v(t,\pm R)=0.
\ee
Let $T:=\min\big(T(u_0),T(u_0^*)\big)\in(0,+\infty]$ be the minimum of the maximal existence times $T(u_0)\in(0,+\infty]$ and $T(u_0^*)\in(0,+\infty]$ of the solutions~$u$ and $v$, respectively. By the maximum principle, both solutions remain nonnegative. Moreover, for every $t\in(0,T(u_0^*))$, the function~$x\mapsto v(t,x)$ is even, continuous, and non-increasing with respect to $|x|$, whence $v^*(t,\cdot)=v(t,\cdot)$. 

The mass concentration comparison principle states that, for many functions $f$, $u(t,\cdot)$ is less concentrated than $v(t,\cdot)$ for all $t\in(0,T)$, written as
\begin{equation}\label{eq:mass_comp}
    u^*(t,\cdot)\preceq v(t,\cdot)
    \quad \hbox{for all } t\in(0,T),
\end{equation}
in the following sense:
\begin{equation}\label{eq:mass_comp2}
    \int_{-r}^{r}u^*(t,x)\,dx
    \le
    \int_{-r}^{r}v(t,x)\,dx
    \quad \hbox{for all } t\in(0,T),\ r\in[0,R).
\end{equation}
When $R<+\infty$, \eqref{eq:mass_comp} follows from \cite{AlvLioTro90} in the linear case $f(u)=\mu u$ with $\mu\in\R$, actually for the analogues of~\eqref{eq:PDE-R},~\eqref{eq:PDE2} and~\eqref{eq:mass_comp2} in a bounded domain of $\R^N$ and in the equimeasurable Euclidean ball, in any dimension $N\ge1$. Inequalities similar to~\eqref{eq:mass_comp}-\eqref{eq:mass_comp2}, also with more general $L^q$ norms instead of~$L^1$ norms, have been shown for equations~\eqref{eq:PDE-R} and~\eqref{eq:PDE2} with convex functions~$f$ such that $f(0)>0$ in~\cite{Ban76} or nonincreasing convex functions $f$ such that $f(0)=0$ in~\cite{DiaGom15}, for linear parabolic equations with quite general and possibly singular zero-order coefficients in~\cite{AlvVol05,AlvVol08,Vol93,VolVol07}, for parabolic equations with degenerate nonlinear diffusion in~\cite{Tra87,Vaz82,Vaz05}, or in~\cite{AlvVol10} for quasilinear parabolic equations with $p$-Laplace type diffusion, zero-order term $c(t,x)|u|^{p-2}u$, and $\partial_t(|u|^{p-2}u)$ instead of $\partial_tu$ in the left-hand side, with $p>2$.

It is therefore natural to ask whether the mass concentration principle~\eqref{eq:mass_comp} holds for semilinear parabolic equations~\eqref{eq:PDE-R} and~\eqref{eq:PDE2} with general functions $f$, in the whole line~$\R$ or in bounded intervals $(-R,R)$. The inequality~\eqref{eq:mass_comp} would give an a priori upper bound of any nonnegative solution by the one emanating from the Schwarz rearranged initial datum. From a biological point of view, it would suggest that the most aggregated initial population has the best chance of persistence among all initial data with the same distribution function. 

However, it turns out that this mass concentration principle is false in general, and we  provide two different counter-examples. The first one shows that~\eqref{eq:mass_comp} fails for~\eqref{eq:PDE-R} and~\eqref{eq:PDE2}, in $\R$ or in bounded intervals with Dirichlet boundary conditions~\eqref{eq:dir}, with simple standard functions~$f:[0,1]\to\R$ satisfying $f(0)=f(1)=0$.

\begin{theorem}\label{thm:frag-fav2}
Let $f:[0,1]\to\R$ be any Lipschitz continuous function satisfying $f(0)=f(1)=0$ and one of the three following  conditions: a)~either~\eqref{hyp:f1a}; b)~or there is $\theta\in(0,1)$ such that $f=0$ on $[0,\theta]\cup\{1\}$ and $f>0$ on $(\theta,1)$; c)~or $f>0$ on $(0,1)$. Then the following conclusions hold. 
\begin{enumerate}
\item [(i)]  There exists a real number $L^*\ge0$ satisfying the following property: for any $L>L^*$, there is $d_0\ge0$ such that, for every $d\ge d_0$ and $x_0\in\R$, the solutions $u$ and $v$ of~\eqref{eq:PDE-R} and~\eqref{eq:PDE2} $($with $R=+\infty$$)$, with respective initial data
$$u_0:=\mathds{1}_{(x_0-L,x_0)\cup(x_0+d,x_0+d+L)}\ \hbox{ and }\ u_0^*=\mathds{1}_{(-L,L)},$$
both have maximal existence time $+\infty$ and range in $[0,1]$, but do not satisfy the mass concentration comparison~\eqref{eq:mass_comp}, that is, there exist $t>0$ and $r>0$ such that
\be\label{eq:nomass}
\int_{-r}^{r}u^*(t,x)\,dx>\int_{-r}^{r}v(t,x)\,dx.
\ee
Moreover, one can choose $L^*=L^*(1)$ in case~a), and $L^*=0$ in case~c) if $\liminf_{s\to0^+}f(s)/s^3>0$.
\item [(ii)] For any real number $R>0$, there is $M_0>0$ with the following property: for any $M\ge M_0$, there are initial data $u_0:=\mathds{1}_E$ with some bounded sets $E\Subset I_R=(-R,R)$, such that, when $f$ is replaced by $M\,f$, the solutions $u$ and~$v$ of~\eqref{eq:PDE-R} and~\eqref{eq:PDE2} with Dirichlet boundary conditions~\eqref{eq:dir}, have maximal existence time~$+\infty$ and range in $[0,1]$, but do not satisfy the mass concentration comparison~\eqref{eq:mass_comp}. More precisely, there exist $t>0$ and $r\in(0,R)$ satisfying~\eqref{eq:nomass}.
\end{enumerate}
\end{theorem}

The previous counter-example shows that the mass concentration principle fails in general. The next counter-example shows that this principle is false in a stronger sense for equations set in $\R$: there are cases when a bounded nonnegative solution $u$ of~\eqref{eq:PDE-R} converges to a positive constant locally uniformly in $\R$ as $t\to+\infty$ while the solution of~\eqref{eq:PDE2} in $\R$ converges to $0$ uniformly in $\R$ as $t\to+\infty$.

\begin{theorem}\label{thm:frag-fav3}
There exist $C^\infty$ functions $f:[0,1]\to\R$ with $f(0)=0$, and initial data $u_0:=\mathds{1}_E$ for some bounded sets $E\subset\R$, such that the solutions $u$ and $v$ of~\eqref{eq:PDE-R} and~\eqref{eq:PDE2} $($with $R=+\infty$$)$ have maximal existence time $+\infty$ and range in $[0,1]$, but do not satisfy the mass concentration comparison~\eqref{eq:mass_comp}. Moreover, $v(t,\cdot)\to0$ uniformly in $\R$ as $t\to+\infty$, while $u(t,\cdot)$ converges to a positive constant locally uniformly in $\R$ as $t\to+\infty$: therefore, for every fixed $r>0$,~\eqref{eq:nomass} holds for all $t$ large enough, and~\eqref{eq:mass_comp} is violated for all $t$ large enough.
\end{theorem}

Theorem~\ref{thm:frag-fav3} is actually a consequence of Theorems~\ref{thm:limitL} and~\ref{thm:frag-fav}. We stated the result with $C^\infty([0,1])$ functions~$f$ to shed light on the fact that even smooth functions $f$ can give counter-examples to~\eqref{eq:mass_comp}. Actually, if one considers any locally Lipschitz continuous function $g:\R_+\to\R$ such that $g(0)=0$, differentiable at~$0$ and satisfying the assumptions~\eqref{hyp:f1a} and~\eqref{hyp:f2} with any $p>3$ and~$A>0$, or with $p\in(2,3]$ and $A>0$ large enough, then the rescaled functions $f$ defined in $[0,1]$ by $f(s)=\frac{g(\alpha\,s)}{\alpha}$ with large $\alpha>0$ are functions for which the conclusion of Theorem~\ref{thm:frag-fav3} holds. In particular, the functions $f$ may not be of class $C^\infty([0,1])$. Typical examples of functions~$f$ for which the conclusion of Theorem~\ref{thm:frag-fav3} holds are then given in $[0,1]$ by:
$$f(s)=C\,s\,(\alpha\,s-\theta)\,(1-\alpha\,s)\,|1-\alpha\,s|^q$$
with any given real numbers $C>0$, $q>0$ and $\theta\in(0,2/(q+4))$ (which ensures ${\int _0^{1/\alpha}\!f(s)ds\!>\!0}$ as one can check) and any $\alpha>0$ large enough (depending on $C$, $q$ and~$\theta$), or by the cubic nonlinearity
$$f(s)=C_1\,s\,(\alpha\,s-\theta)\,(1-\alpha\,s)\ \hbox{ for $s\in[0,1/\alpha]$}$$
with any given $\theta\in(0,1/2)$ and $C_1>0$, and $f(s)\le-\frac{C_2(\alpha\,s-1)^3}{\alpha}$ in $[1/\alpha,1]$ with any $C_2>0$ large enough (depending on $C_1$ and $\theta$) and $\alpha>0$ large enough (depending on $C_1$, $\theta$ and~$C_2$). In Theorem~\ref{thm:frag-fav2}, the functions $f$  can be of class~$C^\infty([0,1])$ as well, but not necessarily.

We point out that, in Theorem~\ref{thm:frag-fav3}, the solutions $u$ of~\eqref{eq:PDE-R} converge to positive constants as $t\to+\infty$ locally uniformly in $\R$, while the solutions $v$ of~\eqref{eq:PDE2} in $\R$ converge to $0$ uniformly in $\R$, whence the mass concentration principle~\eqref{eq:mass_comp} is violated for all $t$ large enough. On the other hand in Theorem~\ref{thm:frag-fav2}-(i),  as revealed by the proof,  both solutions $u$ and $v$ converge to $1$ as $t\to+\infty$ locally uniformly in $\R$, and the mass concentration principle~\eqref{eq:mass_comp} is shown to be violated only for some $t>0$. Lastly, the functions $f$ considered in Theorem~\ref{thm:frag-fav3} could also serve as counter-examples to~\eqref{eq:mass_comp} in bounded intervals $(-R,R)$, after multiplying them by a sufficiently large constant  and rescaling the supports of the initial data~$u_0$.

\vskip 0.4cm
\noindent{\bf{Outline of the paper.}} In Section~\ref{sec2}, we first establish some preliminary $L^\infty$ and $L^1$ estimates for the solutions of~\eqref{eq:PDE} with initial data of the type~\eqref{eq:IC-alpha-L}, by bounding them from above by explicit upper bounds of the solutions of the ordinary differential equation $U'(t)=f(U(t))$ with $U(0)=\alpha$. Section~\ref{sec3} is devoted to the proof of the main results, namely Theorems~\ref{thm:limitL} and~\ref{th:Linfty-psmall} on whether $L_\infty$ is positive or zero, Theorem~\ref{thm:frag-fav} and Corollary~\ref{cor:one-hole} on the effect of fragmentation of the initial datum on the extinction or propagation properties of the solutions of~\eqref{eq:PDE} at large time, and Theorems~\ref{thm:frag-fav2} and~\ref{thm:frag-fav3} on different counter-examples to the mass concentration principle for equations~\eqref{eq:PDE-R} and~\eqref{eq:PDE2} in the whole line, or in bounded intervals $(-R,R)$ with Dirichlet boundary conditions~\eqref{eq:dir}.


\section{Preliminary estimates}\label{sec2}

We start in this section with a series of preparatory steps of the main results.

\begin{lemma}\label{lem:alpha-monot}
Assume that $f$ satisfies~\eqref{hyp:f1a}-\eqref{hyp:f1b}. Then the map $\alpha\mapsto L^*(\alpha)$ is decreasing in $(\theta,+\infty)$.
\end{lemma}

\begin{proof}
We first note that the comparison principle implies that $\alpha\mapsto L^*(\alpha)$ is nonincreasing on $(\theta,+\infty)$. Indeed, if $\theta<\alpha_1<\alpha_2$ and $L>0$, then
\[
u_{\alpha_1}^L(0,x)\le u_{\alpha_2}^L(0,x)\ \text{ for all } x\in\R,
\]
hence
\[
u_{\alpha_1}^L(t,x)\le u_{\alpha_2}^L(t,x)\ \text{ for all } t>0, \  x\in\R.
\]
Therefore, thanks to Theorems~\ref{thm:Zla}-\ref{thm:DM}, if $u_{\alpha_1}^L$ propagates, then so does $u_{\alpha_2}^L$, which yields
\[
L^*(\alpha_2)\le L^*(\alpha_1).
\]

It remains to prove that this inequality is strict. Let $\theta<\alpha_1<\alpha_2$, and argue by contradiction. Assume that
\[
L^*(\alpha_1)=L^*(\alpha_2)=:L.
\]
For this fixed $L>0$, we  consider the one-parameter family $\{\phi_\alpha\}_{\alpha\ge0}$ defined by
\[
\phi_\alpha:=\alpha\,\mathds{1}_{(-L/2,L/2)},
\qquad \alpha\ge0.
\]
This family $\{\phi_\alpha\}_{\alpha\ge0}$ satisfies \eqref{hyp:DuMat}, and each function $\phi_\alpha$ is even. Applying Theorem~\ref{thm:DM} to this family, there exists a threshold value $\alpha_*(L)\in(0,+\infty]$ such that the corresponding solution $u_\alpha^L$ of~\eqref{eq:PDE} with initial datum $\phi_\alpha$ satisfies
\[
\lim_{t\to+\infty}u_\alpha^L(t,\cdot)=
\left\{\begin{array}{lll}
0 & \text{uniformly in } \R, & \text{if } \alpha<\alpha_*(L),\\[1mm]
p & \text{uniformly in } \R, & \text{if } \alpha=\alpha_*(L),\\[1mm]
1 & \text{locally uniformly in } \R, & \text{if } \alpha>\alpha_*(L).
\end{array}\right.
\]
On the other hand, owing to the definition of $L^*(\alpha_i)$ and our assumption $L=L^*(\alpha_i)$, the solution $u_{\alpha_i}^L$ converges to the ground state $p$ for both $i=1,2$.  This implies that $\alpha_*(L)$ is finite (otherwise
all finite values of $\alpha$ would lead to extinction) and 
\[
\alpha_1=\alpha_*(L)=\alpha_2,
\]
which contradicts $\alpha_1<\alpha_2$. Therefore one must have
\[
L^*(\alpha_2)<L^*(\alpha_1)
\quad \text{whenever } \theta<\alpha_1<\alpha_2.
\]
This proves that the map $\alpha\mapsto L^*(\alpha)$ is decreasing in $(\theta,+\infty)$.
\end{proof}

\begin{lemma}\label{lem:smallL1}
Assume that $f$ satisfies~\eqref{hyp:f1a}-\eqref{hyp:f1b}, and define
\[
\gamma:=\sup_{s>0}\frac{f(s)}{s}=\sup_{0<s<1}\frac{f(s)}{s}\in(0,+\infty),
\qquad
m_*:=\frac{\theta\sqrt{\pi}}{e^\gamma}.
\]
If $u_0\in L^1(\R)\cap L^\infty(\R)$, $u_0\ge0$, and $\|u_0\|_{L^1(\R)}\le m_*$, then the corresponding solution of~\eqref{eq:PDE} satisfies
\[
\lim_{t\to+\infty}u(t,\cdot)=0
\quad \text{uniformly in } \R.
\]
\end{lemma}

\begin{proof}
Since $f(s)\le \gamma s$ for all $s\ge0$, the comparison principle yields
\[
0\le u(1,x)\le e^\gamma\,(G(1,\cdot)\ast u_0)(x)
\le \frac{e^\gamma}{\sqrt{4\pi}}\|u_0\|_{L^1(\R)}
\le \frac\theta2
\ \text{ for all } x\in\R,
\]
where
\be\label{defG}
G(t,x):=\frac{e^{-x^2/(4t)}}{\sqrt{4\pi t}},
\quad t>0,\ x\in\R,
\ee
denotes the heat kernel on $\R$. Let $y$ be the solution of $y'(t)=f(y(t))$ with $y(1)=\theta/2$. Since $f<0$ on $(0,\theta)$ and $f(0)=0$, one has $y(t)\to 0$ as $t\to+\infty$. The comparison principle gives
\[
0\le u(t,x)\le y(t)
\ \text{ for all } t\ge1, \ x\in\R,
\]
and the conclusion follows.
\end{proof}

\begin{lemma}\label{lem:ODE-bound}
Assume that $f$ is locally Lipschitz-continuous in $[1,+\infty)$ and satisfies $f(1)=0$ together with~\eqref{hyp:f2} for some $p>1$ and $A>0$. Let $\alpha>1$ and let $U$ solve
$$\left\{\begin{array}{ll}
U'(t)=f(U(t)), & t>0,\vspace{3pt}\\
U(0)=\alpha. & \end{array}\right.$$
Then, for every $t>0$,
\be\label{defB}
U(t)\le B(t):=1+\big(A\,(p-1)\,t\big)^{-\frac1{p-1}}.
\ee
\end{lemma}

\begin{proof}
Set $V(t):=U(t)-1$. Since $U(0)>1$, $f(1)=0$, and $f$ is locally Lipschitz-continuous and nonpositive in $[1,+\infty)$, $U$ is defined in $[0,+\infty)$ and $V(t)>0$ for all $t>0$. Using~\eqref{hyp:f2},
\[
V'(t)=f(U(t))\le -A\,V(t)^p.
\]
Since $p>1$,
\[
\bigl(V(t)^{1-p}\bigr)'=-(p-1)V(t)^{-p}V'(t)\ge A\,(p-1).
\]
Integrating in time gives
\[
V(t)^{1-p}-V(0)^{1-p}\ge A\,(p-1)\,t
\]
for all $t>0$, whence
\[
U(t)=1+V(t)\le 1+\big(A\,(p-1)\,t\big)^{-\frac1{p-1}}.
\]
This provides the desired estimate~\eqref{defB}.
\end{proof}

\begin{proposition}\label{prop:L1-bound}
Assume that $f$ satisfies \eqref{hyp:f1a} and \eqref{hyp:f2} for some $p>1$ and $A>0$, and define $\gamma>0$ as in Lemma~$\ref{lem:smallL1}$. Then, for every $\alpha>1$ and $L>0$, the solution $u^L_\alpha$ of~\eqref{eq:PDE} with initial datum~\eqref{eq:IC-alpha-L} satisfies
\begin{equation}\label{eq:L1-bound}
\|u_\alpha^L(t,\cdot)\|_{L^1(\R)}
\le L\,B(t)+\frac{2e^{\gamma t}}{\sqrt{\pi}}\int_0^t \frac{B(s)}{\sqrt{t-s}}\,ds
\end{equation}
for all $t>0$, where $B$ is defined in~\eqref{defB} and the right-hand side of~\eqref{eq:L1-bound} is understood as~$+\infty$ if the integral diverges, that is, if $p\le2$.
\end{proposition}

\begin{proof}
Let $U$ be the ODE solution from Lemma~\ref{lem:ODE-bound}. By the comparison principle and Lemma~\ref{lem:ODE-bound}, one has, for all $t>0$ and $x\in\R$,
\[
0\le u_\alpha^L(t,x)\le U(t)\le B(t).
\]
Therefore, for all $t>0$,
\[
\int_{-L/2}^{L/2}u_\alpha^L(t,x)\,dx\le L\,B(t).
\]

It remains to control the exterior region. Consider the linear problem on the half-line
\[
\begin{cases}
\partial_t w=\partial_{xx}w+\gamma w, & t>0,\ x>L/2,\\[2pt]
w(t,L/2)=U(t), & t>0,\\[2pt]
w(0,x)=0, & x>L/2.
\end{cases}
\]
Since $f(s)\le \gamma s$ for all $s\ge0$, the function $w$ is a supersolution of \eqref{eq:PDE} on $(L/2,+\infty)$, with the correct boundary and initial domination. Hence
\[
0\le u_\alpha^L(t,x)\le w(t,x)
\quad \text{for all } t>0,\ x>L/2.
\]

Setting $z(t,x):=w(t,x+L/2)$, one has
\[
\begin{cases}
\partial_t z=\partial_{xx}z+\gamma z, & t>0,\ x>0,\\[2pt]
z(t,0)=U(t), & t>0,\\[2pt]
z(0,x)=0, & x>0.
\end{cases}
\]
The standard representation formula on the half-line, see for instance \cite{Can84}, gives
\[
z(t,x)=e^{\gamma t}\int_0^t P(t-s,x)\,e^{-\gamma s}\,U(s)\,ds
\]
for all $t>0$ and $x>0$, where
\[
P(\tau,x)
= -2\frac{\partial}{\partial x}\left(\frac{e^{-x^2/(4\tau)}}{\sqrt{4\pi\tau}}\right)
= \frac{x\,e^{-x^2/(4\tau)}}{2\sqrt{\pi}\,\tau^{3/2}},
\quad \tau>0,\ \ x>0.
\]
Moreover, for all $\tau>0$,
\[
\int_0^{+\infty}P(\tau,x)\,dx
=\int_0^{+\infty}\frac{x\,e^{-x^2/(4\tau)}}{2\sqrt{\pi}\,\tau^{3/2}}\,dx
=\frac{1}{\sqrt{\pi\tau}}.
\]
Thus, for all $t>0$,
\begin{align*}
\int_{L/2}^{+\infty}u_\alpha^L(t,x)\,dx
\le \int_{L/2}^{+\infty}w(t,x)\,dx
& = \int_0^{+\infty}z(t,x)\,dx \\
&= e^{\gamma t}\int_0^{+\infty}\int_0^t P(t-s,x)\,e^{-\gamma s}\,U(s)\,ds\,dx \\
&= e^{\gamma t}\int_0^t e^{-\gamma s}\,U(s)
   \left(\int_0^{+\infty}P(t-s,x)\,dx\right)ds \\
&= \frac{e^{\gamma t}}{\sqrt{\pi}}\int_0^t \frac{e^{-\gamma s}\,U(s)}{\sqrt{t-s}}\,ds\\
&\le \frac{e^{\gamma t}}{\sqrt{\pi}}\int_0^t \frac{B(s)}{\sqrt{t-s}}\,ds.
\end{align*}
The same estimate holds on $(-\infty,-L/2)$ by symmetry. Combining the three regions, we obtain
\[
\|u_\alpha^L(t,\cdot)\|_{L^1(\R)}
\le L\,B(t)+\frac{2e^{\gamma t}}{\sqrt{\pi}}\int_0^t \frac{B(s)}{\sqrt{t-s}}\,ds,
\]
which is exactly \eqref{eq:L1-bound}.
\end{proof}


\section{Proofs of the main results}\label{sec3}

This section is devoted to the proof of the main results described in Section~\ref{sec1}.  We start with the rather straightforward Corollary~\ref{cor:limitL}.

\begin{proof}[Proof of Corollary~\ref{cor:limitL}]
Since $\alpha\mapsto L^*(\alpha)$ is decreasing (by Lemma~\ref{lem:alpha-monot}), if $0<L\le L_\infty$, then $L <  L^*(\alpha)$ for every $\alpha>\theta$, hence extinction follows from Theorems~\ref{thm:Zla}-\ref{thm:DM}, that is,~$u^L_\alpha(t,\cdot)\to0$ as $t\to+\infty$ uniformly in $\R$. This also holds for all $\alpha\in(0,\theta]$, by the parabolic maximum principle. Now, if $L>L_\infty$, then there exists $\alpha_0>\theta$ such that $L>L^*(\alpha)$ for all~$\alpha>\alpha_0$, and Theorems~\ref{thm:Zla}-\ref{thm:DM} entail propagation of the solution $u^L_\alpha$.
\end{proof}

In Subsection~\ref{sec31} we prove Theorem~\ref{thm:limitL}, using in particular Lemma~\ref{lem:smallL1} and Proposition~\ref{prop:L1-bound}. Then, Subsection~\ref{sec32} is concerned with the proof of Theorem~\ref{th:Linfty-psmall}, using a variety of refined subsolutions. In Subsection~\ref{sec33}, we carry out the proofs of Theorem~\ref{thm:frag-fav}, Corollary~\ref{cor:one-hole}  on the effect of fragmentation of the initial datum on the large-time dynamics. Lastly, in Subsection~\ref{ss:mass}, we prove the results on the failure of the mass concentration principle, namely Theorems~\ref{thm:frag-fav2} and~\ref{thm:frag-fav3}.


\subsection{Conditions ensuring $L_\infty>0$: proof of Theorem~\ref{thm:limitL}}\label{sec31}

\begin{proof}[Proof of Theorem~\ref{thm:limitL}]
We use Proposition~\ref{prop:L1-bound} together with the small-$L^1$ criterion from Lemma~\ref{lem:smallL1}.

\smallskip
\noindent
\textit{Case (i):} $p>3$ and arbitrary $A>0$. Fix $p>3$ and $A>0$. With $B$ as in~\eqref{defB}, it follows that, for all $t>0$,
\begin{align*}
    \int_0^t\frac{B(s)}{\sqrt{t-s}}\,ds &
=\
\int_0^t\frac{ds}{\sqrt{t-s}}
+\left(\frac{1}{A\,(p-1)}\right)^{\frac1{p-1}}
\int_0^t\frac{s^{-\frac1{p-1}}}{\sqrt{t-s}}\,ds \\
& =\  
2\sqrt{t}
+
\left(\frac{1}{A\,(p-1)}\right)^{\frac1{p-1}}
t^{\frac12-\frac1{p-1}}
\int_0^1\frac{r^{-\frac1{p-1}}}{\sqrt{1-r}}\,dr.
\end{align*}
Since $p>2$, one has $\frac1{p-1}<1$, and therefore
\[
\int_0^1\frac{r^{-\frac1{p-1}}}{\sqrt{1-r}}\,dr<+\infty.
\]
Since moreover $p>3$, the exponent $\frac12-\frac1{p-1}$ is positive, and therefore
\[
\int_0^t\frac{B(s)}{\sqrt{t-s}}\,ds\ \mathop{\longrightarrow}_{t\to0^+}\ 0.
\]
Choose $t_0>0$ small enough so that
\[
\frac{2\,e^{\gamma t_0}}{\sqrt{\pi}}\int_0^{t_0}\frac{B(s)}{\sqrt{t_0-s}}\,ds<\frac{m_*}{2},
\]
where $m_*$ is given by Lemma~\ref{lem:smallL1}. We may then choose $L_0>0$ such that
\[
L_0B(t_0)<\frac{m_*}{2}.
\]
By Proposition~\ref{prop:L1-bound}, every solution $u_\alpha^{L_0}$ with $\alpha>1$ then satisfies
\[
\|u_\alpha^{L_0}(t_0,\cdot)\|_{L^1(\R)}<m_*.
\]
Applying Lemma~\ref{lem:smallL1} with the new initial time $t_0$ yields extinction of $u_\alpha^{L_0}$ for every $\alpha>1$. Therefore
\[
L^*(\alpha)>L_0\ \text{ for all } \alpha>1,
\]
and thus $L_\infty\ge L_0>0$. 

\medskip
\noindent
\textit{Case (ii):} $p>2$ and large $A>0$. Let
\[
t_1:=\frac{1}{A\,(p-1)},
\quad \text{so that }B(t_1)=2.
\]
Applying Proposition~\ref{prop:L1-bound} with $L=2/A$ and any $\alpha>1$ gives
\[
\|u_\alpha^{2/A}(t_1,\cdot)\|_{L^1(\R)}
\le\frac4A+\frac{2\,e^{\gamma t_1}}{\sqrt{\pi}}\int_0^{t_1}\frac{B(s)}{\sqrt{t_1-s}}\,ds.
\]
Since
\[
B(s)=1+\left(\frac{t_1}{s}\right)^{\frac1{p-1}},
\]
a change of variables $s=t_1r$ yields
\[
\int_0^{t_1}\frac{B(s)}{\sqrt{t_1-s}}\,ds
=\sqrt{t_1}\underbrace{\int_0^1\frac{1+r^{-\frac1{p-1}}}{\sqrt{1-r}}\,dr}_{=:B_p}.
\]
Because $p>2$, the  integral $B_p$ is finite. Hence
\[
\|u_\alpha^{2/A}(t_1,\cdot)\|_{L^1(\R)}
\le \frac4A+\frac{2\,e^{\gamma t_1}B_p\,\sqrt{t_1}}{\sqrt{\pi}}.
\]
Using the definition of $t_1=\frac1{A\,(p-1)}$, we observe that the right-hand side tends to $0$ as~${A\to+\infty}$. We may therefore choose $A^*>0$, depending only on $\gamma$ and $p$, such that for every $A>A^*$,
\[
\frac4A+\frac{2\,e^{\gamma t_1}\,B_p\,\sqrt{t_1}}{\sqrt{\pi}}<m_*.
\]
For such an $A$, Lemma~\ref{lem:smallL1} implies that $u_\alpha^{2/A}$ extinguishes for every $\alpha>1$. Consequently,
\[
L^*(\alpha)>\frac2A\ \text{ for all } \alpha>1,
\]
and therefore $L_\infty\ge\frac2A>0$. The proof of Theorem~\ref{thm:limitL} is thereby complete.
\end{proof}


\subsection{Conditions ensuring $L_\infty=0$: proof of Theorem~\ref{th:Linfty-psmall}}\label{sec32}

We distinguish the proofs of case (i) and case (ii).

\begin{proof}[Proof of Theorem \ref{th:Linfty-psmall}-(i)] First,  if $L=\frac{m_*}{\alpha}$ in~\eqref{eq:IC-alpha-L} with $m_*>0$ as in Lemma~\ref{lem:smallL1}, then extinction occurs from the small-$L^1$ criterion of Lemma~\ref{lem:smallL1}. Hence
$$L^*(\alpha)>\frac{m_*}{\alpha},$$
for every $\alpha>\theta$. 

Next, since $f$ satisfies \eqref{hyp:f1a} and \eqref{hyp:f3}
 for some $0\leq p\leq 1$, we can define
$$
\displaystyle \rho:=-\inf_{s>0}\frac{f(s)}{s}\in(0,+\infty).
$$
We now fix any $R>L^*(1)$, for instance $R=L^*(1)+1$, and define
$$
M:=\sqrt{4\pi}\,e^{\rho+\frac{(R+1)^2}{16}}>1>\theta.
$$
We show below that the initial size $L=\frac{M}{\alpha}$ in~\eqref{eq:IC-alpha-L} leads to propagation as soon as $\alpha \geq M$.

So, with $\alpha\ge M$ and $L=\frac{M}{\alpha}\in(0,1]$, let $\underline u_\alpha^L$ be the solution of the linear Cauchy problem
\begin{equation}\label{eq:linear-sub}
\begin{cases}
\partial_t \underline u=\partial_{xx}\underline u-\rho\,\underline u,
& t>0,\ x\in\R,\\[2pt]
\underline u(0,x)=\alpha\,\mathds{1}_{(-L/2,L/2)}(x), & x\in\R.
\end{cases}
\end{equation}
Since $f(s)\ge -\rho s$ for all $s\ge0$, the function $\underline u_\alpha^L$ is a subsolution of \eqref{eq:PDE}, and hence
\be\label{inequ}
u_\alpha^L(t,x)\ge \underline u_\alpha^L(t,x)\ 
\text{ for all } t\ge0, \ x\in\R.
\ee
Moreover, for all $t>0$ and $x\in\R$, one has
\begin{equation}\label{eq:linear-sub-explicit}
\underline u_\alpha^L(t,x)
=\alpha\,e^{-\rho t}\int_{-L/2}^{L/2}G(t,x-y)\,dy,
\end{equation}
where $G$ denotes the heat kernel on $\R$, defined in~\eqref{defG}.

If $|x|\le R/2$ and $|y|\le L/2$, then $|x-y|\le (R+L)/2\le (R+1)/2$, hence
\[
G(1,x-y)\ge \frac{1}{\sqrt{4\pi}}\exp\left(-\frac{(R+1)^2}{16}\right).
\]
Therefore,
\[
\underline u_\alpha^L(1,x)
\ge \alpha\,e^{-\rho}\,\frac{L}{\sqrt{4\pi}}
\exp\left(-\frac{(R+1)^2}{16}\right)
\text{ for all } |x|\le R/2.
\]
From the above choices and~\eqref{inequ}, we obtain
\[
u_\alpha^L(1,x)\ge1
\text{ for all } |x|\le R/2.
\]
Hence,
\[
u_\alpha^L(1,\cdot)\ge u_1^R(0,\cdot).
\]
Since $R>L^*(1)$, Theorems~\ref{thm:Zla}-\ref{thm:DM} imply that the solution of~\eqref{eq:PDE} with initial datum $u_1^R(0,\cdot)$ propagates. The comparison principle then yields propagation of $u_\alpha^L$ with~$L=\frac{M}{\alpha}$, that is,
\[
L^*(\alpha)<\frac M\alpha\  \text{ for all } \alpha \geq M,
\]
which concludes the proof.
\end{proof}

\begin{proof}[Proof of Theorem \ref{th:Linfty-psmall}-(ii)]
Throughout the proof, we assume that $f$ satisfies~\eqref{hyp:f1a} and~\eqref{hyp:f3} with $1<p<3$. The problem under consideration shares some similarities with the following heat equation with absorption
$$
U_t=\Delta U-U^p, \quad t>0,\ x\in \R^N,
$$
which has attracted a lot of attention \cite{Bre-Pel-Ter-86,Esc-Kav-Mat-95,Gmi-Ver-84,Her-99,Kam-Pel-85,Kam-Pel-85-bis}. In dimension $N=1$, when looking at positive and even-in-$x$ self-similar solutions of the form
\begin{equation}\label{ansatz}
W(t,x):=\frac{1}{(t_0+t)^{\frac{1}{p-1}}}\varphi\left(\frac{x}{\sqrt{t_0+t}}\right), \quad t_0>0,\ t\geq 0,\ x\in \R,
\end{equation}
one reaches the ODE Cauchy problem for $\varphi=\varphi(z)$:
\be\label{pb-ode}
\left\{\baa{l}
\varphi''+\frac 12 z\varphi' +\frac{1}{p-1}\varphi=\varphi^p,\quad z\in \R,\vspace{3pt}\\
\varphi(0)=h,\vspace{3pt}\\
\varphi'(0)=0,\eaa\right.
\ee
where $h>0$ can be viewed as a shooting parameter. When $1<p<3$,  Brezis, Peletier and Terman \cite{Bre-Pel-Ter-86} exhibited a very singular solution to the heat equation with absorption which is based on the following tool.

\begin{lemma}[\cite{Bre-Pel-Ter-86}]\label{lem:bre-pel-ter} Assume $1<p<3$. Then there is a unique $h>0$ such that the (even) solution to \eqref{pb-ode} is global, and satisfies
$$
\varphi>0 \text{ on } \R, \quad \lim_{z\to +\infty} z^{\frac{2}{p-1}}\varphi(z)=0.
$$
Furthermore, $\varphi'<0$ on $(0,+\infty)$, and
\begin{equation}
    \label{asympto-varphi}
\varphi(z)\sim C e^{-\frac 1 4 z^2}z^{\frac{3-p}{p-1}}\ \text{ as } z\to +\infty,
\end{equation}
for some known constant $C>0$.  
\end{lemma}

We denote by $\varphi$ the solution provided by Lemma \ref{lem:bre-pel-ter} and recall $h=\varphi(0)>0$. We fix
\begin{equation}
    \label{varepsilon}
    0<\varepsilon<\frac{1}{p(p-1)}.
\end{equation}
Let us denote
$$
R:=L^*(1)+1,
$$
and
$$
M:=-\min _{s\in[0,1]}f(s)\in (0,+\infty).
$$
We emphasize that $R$ and $M$ depend only on $f|_{[0,1]}$, and are thus independent of $a>0$ appearing in \eqref{hyp:f3}. We fix $T>0$ large enough so that
 \begin{equation}\label{T-large}
 \varphi\left(\frac{R}{2\sqrt{T}}\right)\geq \frac {h}2.
\end{equation}
We now fix $a_*>0$ small enough so that 
\begin{equation}
    \label{a-etoile-1}
    \forall\,0<a<a_*,\ \ \ \frac{1}{a^{\varepsilon p}}\geq M,
\end{equation}
together with
 \begin{equation}
     \label{a-etoile-2}
 \forall\,0<a<a_*,\ \ \ \frac{h}{2\,a^{\frac{1}{p-1}}(1+T)^{\frac{1}{p-1}}}-\frac{1}{a^{\frac 1 p +\varepsilon}}\geq 1.
 \end{equation}
This is possible because, by \eqref{varepsilon}, the first term in \eqref{a-etoile-2} dominates the second one as $a\to0^+$.

From now on, we assume that $f$ satisfies \eqref{hyp:f3} with
$$0<a<a_*.$$
Let us fix $L>0$. Our goal is to show that the solution $u_\alpha^L$ starting from
\[
u_\alpha^L(0,x)=\alpha\,\mathds{1}_{(-L/2,L/2)}(x)
\]
propagates for $\alpha$ large enough. From the definition of $M>0$, \eqref{hyp:f3}, and \eqref{a-etoile-1}, one can check that 
\begin{equation}\label{ineg-f-a}
    \forall\,s\geq 0,\ \ \ f(s)\geq -a\left(s+ \frac{1}{a^{\frac{1}{p}+\varepsilon}} \right)^p.
\end{equation}
Indeed, if $0\leq s\leq 1$, then
\[
    f(s)\geq -M\geq -a^{-\varepsilon p}
    =
    -a\left(\frac{1}{a^{\frac1p+\varepsilon}}\right)^p
    \geq
    -a\left(s+\frac{1}{a^{\frac1p+\varepsilon}}\right)^p,
\]
while, if $s>1$, then \eqref{hyp:f3} gives
\[
    f(s)\geq -as^p
    \geq
    -a\left(s+\frac{1}{a^{\frac1p+\varepsilon}}\right)^p.
\]
By comparison,
\[
u_\alpha^L\geq v-\frac{1}{a^{\frac{1}{p}+\varepsilon}}
\]
in $(0,+\infty)\times\R$, where 
$$
v_t=v_{xx}-a v^p, \quad
v_0=\alpha\,\mathds{1}_{(-L/2,L/2)}+\frac{1}{a^{\frac{1}{p}+\varepsilon}}.
$$
Indeed, $u_\alpha^L+a^{-\frac1p-\varepsilon}$ is a supersolution of the equation solved by $v$, thanks to \eqref{ineg-f-a}.

Next, since $W$ defined in~\eqref{ansatz}-\eqref{pb-ode} is a solution of $W_t=W_{xx}-W^p,$
one immediately checks, by scaling the amplitude, that
\[
w(t,x):=\frac{1}{a^{\frac{1}{p-1}}}W(t,x)
=
\frac{1}{a^{\frac{1}{p-1}}(t_0+t)^{\frac{1}{p-1}}}\,
\varphi\left(\frac{x}{\sqrt{t_0+t}}\right)
\]
satisfies
\[
w_t=w_{xx}-aw^p.
\]
Indeed, if $\beta=1/(p-1)$, then
\[
w_t-w_{xx}
=
a^{-\beta}(W_t-W_{xx})
=
-a^{-\beta}W^p
=
-aw^p.
\]

For $w(0,\cdot)\leq v_0$ to hold, it is sufficient, recalling that $\varphi$ is radially decreasing, to have
\begin{equation}
    \label{cond1}
    \frac{h}{a^{\frac{1}{p-1}}t_0^{\frac{1}{p-1}}}\leq \alpha+\frac{1}{a^{\frac{1}{p}+\varepsilon}}
\end{equation}
and 
\begin{equation}
    \label{cond2}
    \frac{1}{a^{\frac{1}{p-1}}t_0^{\frac{1}{p-1}}}\,\varphi\left(\frac{L}{2\sqrt{t_0}}\right)\leq \frac{1}{a^{\frac{1}{p}+\varepsilon}}.
\end{equation}
Indeed, \eqref{cond1} gives the comparison inside $(-L/2,L/2)$, while \eqref{cond2} gives it outside $(-L/2,L/2)$ by monotonicity of $\varphi$. From \eqref{asympto-varphi}, we can choose $0<t_0<1$ small enough so that \eqref{cond2} holds, and then $\alpha>0$ large enough so that \eqref{cond1} holds. For these choices, we deduce from the comparison principle that
$$
u_\alpha^L\geq v-\frac{1}{a^{\frac{1}{p}+\varepsilon}}\geq w-\frac{1}{a^{\frac{1}{p}+\varepsilon}}
$$
in $(0,+\infty)\times\R$. In particular, since $\varphi$ is radially decreasing, we have, for any $|x|\leq \frac R 2$,
\begin{eqnarray*}
    u_\alpha^L(T,x)&\geq &\frac{1}{a^{\frac{1}{p-1}}(t_0+T)^{\frac{1}{p-1}}}\,\varphi\left(\frac{R}{2\sqrt{t_0+T}}\right)-\frac{1}{a^{\frac 1 p +\varepsilon}}
    \\
    &\geq & \frac{1}{a^{\frac{1}{p-1}}(1+T)^{\frac{1}{p-1}}}\,\varphi\left(\frac{R}{2\sqrt{T}}\right)-\frac{1}{a^{\frac 1 p +\varepsilon}}\\
    &\geq & \frac{1}{a^{\frac{1}{p-1}}(1+T)^{\frac{1}{p-1}}}\times\frac h 2-\frac{1}{a^{\frac 1 p +\varepsilon}}\geq 1,
\end{eqnarray*}
where we have used \eqref{T-large} and \eqref{a-etoile-2}. Thus, at time $T$, the solution $u_\alpha^L$ is above $1$ on the interval $(-R/2,R/2)$. Since $R>L^*(1)$, the solution starting from $\mathds{1}_{(-R/2,R/2)}$ propagates. By comparison, $u_\alpha^L$ propagates as well. Since $L>0$ was arbitrary, this proves that $L_\infty=0$.
\end{proof}


\subsection{Fragmentation results: proofs of Theorem~\ref{thm:frag-fav} and Corollary~\ref{cor:one-hole}}\label{sec33}

We begin with the following useful convolution estimate.

\begin{lemma}\label{lem:osc}
Let $K\in W^{1,1}(\R)$, $h>0$, and~$\{J_k\}_{k\in\mathbb{Z}}$ be a family of pairwise disjoint intervals of length $h$ such that $\R=\bigcup_{k\in\mathbb{Z}}\overline{J_k}$. Assume that $q\in L^\infty(\R)$ satisfies
\[
\int_{J_k} q(y)\,dy=0
\quad \text{for every } k\in\mathbb{Z}.
\]
Then
\[
\|K\ast q\|_{L^\infty(\R)}\le h\,\|q\|_{L^\infty(\R)}\,\|K'\|_{L^1(\R)}.
\]
\end{lemma}

\begin{proof}
Fix $x\in\R$. For each $k\in\mathbb{Z}$, set
\[
M_k(x):=\frac1h\int_{J_k}K(x-z)\,dz.
\]
Since $\displaystyle\int_{J_k}q(y)\,dy=0$, there holds
\[
\int_{J_k}K(x-y)\,q(y)\,dy
=\int_{J_k}\bigl(K(x-y)-M_k(x)\bigr)\,q(y)\,dy.
\]
Now, for all $y\in J_k$,\begin{align*}
|K(x-y)-M_k(x)|
& = \frac1h\left| \int_{J_k}(K(x-y)-K(x-z))\,dz\right|\\
&\le \frac1h\int_{J_k}|K(x-y)-K(x-z)|\,dz\\
&\le \frac1h\int_{J_k}\int_{J_k}|K'(x-\xi)|\,d\xi\,dz
=\int_{J_k}|K'(x-\xi)|\,d\xi.
\end{align*}
Hence,
\begin{align*}
    \left|\int_{J_k}K(x-y)\,q(y)\,dy\right|
 & \le  \|q\|_{L^\infty(\R)} \int_{J_k} |K(x-y)-M_k(x)| \, dy \\
 & \le h\,\|q\|_{L^\infty(\R)}\int_{J_k}|K'(x-\xi)|\,d\xi.
\end{align*}
Summing over $k$ gives
\[
|(K\ast q)(x)|\le h\,\|q\|_{L^\infty(\R)}\int_\R |K'(x-\xi)|\,d\xi
= h\,\|q\|_{L^\infty(\R)}\,\|K'\|_{L^1(\R)}.
\]
The conclusion follows.
\end{proof}

We are now in the position to prove Theorem~\ref{thm:frag-fav}.

\begin{proof}[Proof of Theorem~\ref{thm:frag-fav}]
Choose any real numbers $R_1$ and $R_2$ such that
\[
0<R_1<L_\infty<R_2.
\]
Since $L^*(\eta)\to L_\infty$ as $\eta\to+\infty$, there exists $\beta>\theta$ such that
\[
0<L^*(\beta)<R_2.
\]
Set
\be\label{def:alpha}
\alpha:=\frac{(\beta+1)R_2}{R_1}>\beta+1>1,
\qquad
I:=\left(-\frac{R_2}{2},\frac{R_2}{2}\right),
\qquad
\bar u_0:=(\beta+1)\,\mathds{1}_I.
\ee
Since $R_1<L_\infty< L^*(\alpha)$, the solution $u^{R_1}_\alpha$ of~\eqref{eq:PDE} with initial datum
\[
u^{R_1}_\alpha(0,\cdot)=\alpha\,\mathds{1}_{(-R_1/2,R_1/2)}
\]
extinguishes.

It remains to prove propagation for a sufficiently fine fragmentation of the same mass. Fix $\delta>0$ such that
\[
R_2-2\delta>L^*(\beta)>0,
\]
and set
\[
I_\delta:=\left(-\frac{R_2}{2}+\delta,\frac{R_2}{2}-\delta\right)\quad\hbox{and}\quad
M_\alpha:=\max_{s\in[0,\alpha]}|f(s)|.
\]
Let $G$ denote the heat kernel on $\R$, given as in~\eqref{defG}. Since $I_\delta\Subset I$ and thus
$$G(t,\cdot)\ast \bar u_0=(\beta+1)\,G(t,\cdot)\ast\mathds{1}_I\to\beta+1\ \hbox{ uniformly in $I_\delta$ as $t\to0^+$},$$
we may choose $t_0>0$ such that
\begin{equation}\label{eq:bar-u0-close}
(G(t_0,\cdot)\ast \bar u_0)(x)\ge \beta+\frac12\ 
\text{ for all } x\in I_\delta,
\end{equation}
and simultaneously
\begin{equation}\label{eq:choose-t0}
    t_0M_\alpha\le \frac14.
\end{equation}

For $n\in\N\setminus\{0\}$, let
\[
 u_{0,n}:=\alpha\,\mathds{1}_{E_n(R_1,R_2)},
\]
with $E_n(R_1,R_2)$ given in~\eqref{eq:En}. We recall that the Lebesgue measure of $E_n(R_1,R_2)$ is equal to $R_1$, that is, the length of the interval $(-R_1/2,R_1/2)$. In other words, the functions~$u^{R_1}_\alpha(0,\cdot)$ and $u_{0,n}$ have the same distribution function. On each cell
\[
J_{n,k}:=\left(-\frac{R_2}{2}+\frac{kR_2}{n},\ -\frac{R_2}{2}+\frac{(k+1)R_2}{n}\right),\quad k\in\mathbb{Z},
\]
one has
\[
\int_{J_{n,k}}\bigl(u_{0,n}(y)-\bar u_0(y)\bigr)\,dy
=\alpha\,\frac{R_1}{n}-(\beta+1)\,\frac{R_2}{n}=0
\]
for $k=0,\dots,n-1$. Furthermore, if $k\in\mathbb{Z}\setminus\{0,1,\cdots,n-1\}$, then $u_{0,n}=\bar u_0=0$ in $J_{n,k}$, whence $\int_{J_{n,k}}\bigl(u_{0,n}(y)-\bar u_0(y)\bigr)\,dy=0$. Applying Lemma~\ref{lem:osc} with $K=G(t_0,\cdot)$, $h=R_2/n$, the intervals $\{J_{n,k}\}_{k\in\mathbb{Z}}$, and $q_n:=u_{0,n}-\bar u_0$, we get
\[
\|G(t_0,\cdot)\ast q_n\|_{L^\infty(\R)}
\le \frac{R_2}{n}\,\alpha\,\|\partial_x G(t_0,\cdot)\|_{L^1(\R)}.
\]
Since $\|\partial_x G(t_0,\cdot)\|_{L^1(\R)}=1/\sqrt{\pi t_0}$ and $G(t_0,\cdot)\ast q_n$ is actually continuous in $\R$, we can fix an integer $n\ge2$ large enough such that
\begin{equation}\label{eq:osc-error}
\forall\,x\in\R,\quad|(G(t_0,\cdot)\ast q_n)(x)|\le \frac14.
\end{equation}

Let  $u_n$ be the solution of \eqref{eq:PDE} with initial datum $u_{0,n}$. Since $0\le u_{0,n}\le \alpha$ (with $\alpha>1$) and $f(s)<0$ for $s>1$, the constant function $\alpha$ is a supersolution. Hence
\[
0\le u_n(t,x)\le \alpha
\quad \text{for all } t\ge0,\ x\in\R.
\]
In particular,
\[
f(u_n(t,x))\ge -M_\alpha
\quad \text{for all } t\ge0,\ x\in\R.
\]
Finally, let $\underline u_n$ be the solution of
\[
\begin{cases}
\partial_t \underline{u}_n=\partial_{xx}\underline{u}_n-M_\alpha,
& t>0,\ x\in\R,\\[2pt]
\underline{u}_n(0,x)=u_{0,n}(x), & x\in\R.
\end{cases}
\]
By the comparison principle,
\[
u_n(t,x)\ge \underline u_n(t,x)
\quad \text{for all } t\ge0,\ x\in\R.
\]
Moreover,
\[
\underline u_n(t,x)=(G(t,\cdot)\ast u_{0,n})(x)-tM_\alpha\quad \text{for all } t>0,\ x\in\R,
\]
whence
\begin{align*}
    u_n(t_0,x) & \ge \underline{u}_n(t_0,x)\\
    & = (G(t_0,\cdot)\ast u_{0,n})(x)-t_0M_\alpha \\
    & =  (G(t_0,\cdot)\ast \bar u_0)(x)  + (G(t_0,\cdot)\ast q_n)(x)   -t_0M_\alpha
\end{align*}
for all $x\in\R$. Using \eqref{eq:bar-u0-close}, \eqref{eq:choose-t0}, and \eqref{eq:osc-error}, we infer that
\[
u_n(t_0,x)
\ge \left(\beta+\frac12\right)-\frac14-\frac14=\beta
\quad \text{for all } x\in I_\delta.
\]
Hence
\[
u_n(t_0,\cdot)\ge \beta\,\mathds{1}_{I_\delta}.
\]
Since $I_\delta$ is an interval of length $|I_\delta|=R_2-2\delta>L^*(\beta)$, Theorems~\ref{thm:Zla}-\ref{thm:DM} imply that the solution with initial datum $\beta\,\mathds{1}_{I_\delta}$ propagates. By comparison, $u_n$ then propagates as well. Remembering that $u_n$ solves~\eqref{eq:PDE} with initial datum $\alpha\,\mathds{1}_{E_n(R_1,R_2)}$, and that~$u^{R_1}_\alpha$ solves~\eqref{eq:PDE} with the equally distributed initial datum $\alpha\,\mathds{1}_{(-R_1/2,R_1/2)}$ and extinguishes, the proof of Theorem~\ref{thm:frag-fav} is thereby complete.
\end{proof}

\begin{proof}[Proof of Corollary~\ref{cor:one-hole}]
Let $f$, $0<R_1<L_\infty<R_2$, $\alpha>\theta$ and the integer $n\ge2$ be given as in Theorem~\ref{thm:frag-fav}.  
For $j=0,\dots,n-1$, set
\[
c_j:=-\frac{R_2}{2}+\frac{jR_2}{n},
\qquad
d_j:=c_j+\frac{R_1}{n}.
\]
For each $k=0,\dots,n-1$, define
\[
F^{(k)}
:=
\left(\bigcup_{j=0}^{k-1}(c_j,d_j)\right)
\cup
\left(c_k,\ c_k+\frac{(n-k)R_1}{n}\right),
\]
with the convention that the first union is empty when $k=0$.

We note that for every $k=0,\dots,n-1$,
$$|F^{(k)}|=k\frac{R_1}{n}+\frac{(n-k)R_1}{n}=R_1.$$
Moreover,
\[
F^{(0)}=\left(-\frac{R_2}{2},-\frac{R_2}{2}+R_1\right),
\]
which is an interval of length $R_1$, and
\[
F^{(n-1)}=\bigcup_{j=0}^{n-1}(c_j,d_j)=E_n(R_1,R_2).
\]
Additionally, for $k=0,\dots,n-2$, $F^{(k+1)}$ is obtained from $F^{(k)}$ by splitting the connected component $\big(c_k,\,c_k+\frac{(n-k)R_1}{n}\big)$ into two disjoint components $(c_k,d_k)$ and $\big(c_{k+1},\,c_{k+1}+\frac{(n-k-1)R_1}{n}\big)$.

\medskip

By Theorem~\ref{thm:frag-fav}, the solution of~\eqref{eq:PDE} with initial datum $\alpha\,\mathds{1}_{F^{(n-1)}}$ propagates. On the other hand, since $F^{(0)}$ is an interval of length $R_1<L_\infty$, the solution of~\eqref{eq:PDE} with initial datum $\alpha\,\mathds{1}_{F^{(0)}}$ extinguishes.

Consider now the finite set of indices
\[
\mathcal K:=\big\{k\in\{0,\dots,n-1\}:\ \text{the solution of~\eqref{eq:PDE} starting from }\alpha\,\mathds{1}_{F^{(k)}}\text{ does not extinguish}\big\}.
\]
This set is nonempty because $n-1\in\mathcal K$. Let $k_*$ be its smallest element. Since the solution with initial datum  $\alpha\,\mathds{1}_{F^{(0)}}$ extinguishes, one has $k_*\ge1$. Set
\[
\mathcal{E}:=F^{(k_*-1)},
\qquad
\mathcal{F}:=F^{(k_*)}.
\]
Then $|\mathcal{E}|=|\mathcal{F}|=R_1$, and the set $\mathcal{F}$ is obtained from $\mathcal{E}$ by splitting one connected component of $\mathcal{E}$ into two disjoint components. These sets are of the type described in the statement of Corollary~\ref{cor:one-hole}, with $k:=k_*$, $a_i:=c_{i-1}$, $b_i:=d_{i-1}$ for $i=1,\dots,k-1$, $a_k:=c_{k-1}$, $b_k:=c_{k-1}+\frac{(n-k+1)R_1}{n}$, $b'_k:=d_{k-1}$, $a'_{k+1}:=c_k$, and $b'_{k+1}:=c_k+\frac{(n-k)R_1}{n}$. Lastly, the solution of~\eqref{eq:PDE} with initial datum $\alpha\,\mathds{1}_{\mathcal{E}}$ extinguishes, while the solution of~\eqref{eq:PDE} with initial datum $\alpha\,\mathds{1}_{\mathcal{F}}$ does not. This proves the corollary.
\end{proof}


\subsection{Two counter-examples to the mass concentration principle: proofs of Theorems~\ref{thm:frag-fav2} and~\ref{thm:frag-fav3}}\label{ss:mass}

\begin{proof}[Proof of Theorem~\ref{thm:frag-fav2}] 
Throughout the proof, let $f:[0,1]\to\R$ be any Lipschitz-continuous function satisfying $f(0)=f(1)=0$ and one of the three following  conditions: a)~\eqref{hyp:f1a} holds; b)~there is $\theta\in(0,1)$ such that $f=0$ on $[0,\theta]\cup\{1\}$ and $f>0$ on $(\theta,1)$; c)~$f>0$ on~$(0,1)$. 

\vskip 0.4cm
\noindent{}(i) From~\cite{AroWei78,FifMcL77,Kan-64}, in cases~a) and~b), there is a unique real number $c^*$, positive, such that~\eqref{eq:PDE-R}  (with $R=+\infty$)  admits traveling front solutions of the type $\varphi(x-c^*t)$, with $\varphi:\R\to(0,1)$ such that $\varphi(-\infty)=1$ and $\varphi(+\infty)=0$. In case~c), there is $c^*>0$ such that such solutions $\varphi(x-ct)$ exist if and only if $c\ge c^*$. In all cases, there is $L^*  \ge   0$ such that, for any $L>L^*$, the solution $U$ of~\eqref{eq:PDE-R} with initial datum
$$U_0:=\mathds{1}_{(-L/2,L/2)}$$
has maximal existence time $+\infty$, ranges in $[0,1]$, converges to $0$ as $x\to\pm\infty$ for each fixed $t>0$, and converges to $1$ as $t\to+\infty$ locally uniformly in $x\in\R$. Furthermore,
$$\forall\,0\le\gamma<c^*<\gamma',\quad\max_{[-\gamma t,\gamma t]}|U(t,\cdot)-1|\mathop{\longrightarrow}_{t\to+\infty}0,\ \ \max_{(-\infty,-\gamma't]\cup[\gamma't,+\infty)}U(t,\cdot)\mathop{\longrightarrow}_{t\to+\infty}0,$$
see~\cite{AroWei78,FifMcL77}. In case~a), $L^*$ is any real number in $[L^*(1),+\infty)$ with the notation of Theorem~\ref{thm:Zla}; in case~b), $L^*$ is also any large enough positive real number; in case~c), one has $L^*>0$ in general, but one can take $L^*=0$ if $\liminf_{s\to0^+}f(s)/s^3>0$, see~\cite{AroWei78}.

Consider now any $L>L^*$. In addition to the function $U$ of the previous paragraph, let~$v$ be the solution of~\eqref{eq:PDE-R} with initial datum
$$v_0:=\mathds{1}_{(-L,L)}.$$
From the previous paragraph, $v$ has maximal existence time $+\infty$, ranges in $[0,1]$, converges to $0$ as $x\to\pm\infty$ for each $t>0$, and
$$\forall\,0\le\gamma<c^*<\gamma',\quad\max_{[-\gamma t,\gamma t]}|v(t,\cdot)-1|\mathop{\longrightarrow}_{t\to+\infty}0,\ \ \max_{(-\infty,-\gamma't]\cup[\gamma't,+\infty)}v(t,\cdot)\mathop{\longrightarrow}_{t\to+\infty}0.$$
It follows that
$$\int_{-3c^*t/4}^{3c^*t/4}U(t,x)\,dx\sim\frac{3c^*t}{2}\ \hbox{ and }\ \int_{-3c^*t/2}^{3c^*t/2}v(t,x)\,dx\sim 2c^*t\ \ \hbox{ as }t\to+\infty.$$
One can then fix in the sequel a real number $T>0$ such that
\be\label{defT}
\int_{-3c^*T/4}^{3c^*T/4}U(T,x)\,dx>\frac12\int_{-3c^*T/2}^{3c^*T/2}v(T,x)\,dx.
\ee

Choose now $d_0\ge 0$ such that
\be\label{defd0}
\frac{3c^*T}{2}\le L+d_0.
\ee
Consider then any $d\ge d_0$ and $x_0\in\R$, and let $u$ be the solution of~\eqref{eq:PDE-R} with initial datum
\be\label{defu0}
u_0:=\underbrace{\mathds{1}_{(x_0-L,x_0)}}_{=U_0(\cdot-x_0+L/2)}+\underbrace{\mathds{1}_{(x_0+d,x_0+d+L)}}_{=U_0(\cdot-x_0-d-L/2)}\!\!.
\ee
Notice that $u_0^*=\mathds{1}_{(-L,L)}=v_0$. There holds
$$0\le \max\big(U_0(\cdot-x_0+L/2),U_0(\cdot-x_0-d-L/2)\big)\le u_0\le1\ \hbox{ in $\R$},$$
whence
$$0\le \max\big(U(t,x-x_0+L/2),U(t,x-x_0-d-L/2)\big)\le u(t,x)\le1$$
for all $(t,x)\in(0,+\infty)\times\R$ from the maximum principle. Together with~\eqref{defT}, one gets that
\be\label{ineqint1}\begin{array}{rcl}
\displaystyle\int_{-3c^*T/4+x_0-L/2}^{3c^*T/4+x_0-L/2}u(T,x)\,dx & \ge & \displaystyle\int_{-3c^*T/4+x_0-L/2}^{3c^*T/4+x_0-L/2}U(T,x-x_0+L/2)\,dx\vspace{3pt}\\
& = & \displaystyle\int_{-3c^*T/4}^{3c^*T/4}U(T,x)\,dx\ > \ \frac12\int_{-3c^*T/2}^{3c^*T/2}v(T,x)\,dx\end{array}
\ee
and
\be\label{ineqint2}\begin{array}{rcl}
\displaystyle\int_{-3c^*T/4+x_0+d+L/2}^{3c^*T/4+x_0+d+L/2}u(T,x)\,dx & \ge & \displaystyle\int_{-3c^*T/4+x_0+d+L/2}^{3c^*T/4+x_0+d+L/2}U(T,x-x_0-d-L/2)\,dx\vspace{3pt}\\
& = & \displaystyle\int_{-3c^*T/4}^{3c^*T/4}U(T,x)\,dx\ > \ \frac12\int_{-3c^*T/2}^{3c^*T/2}v(T,x)\,dx.\end{array}
\ee
Since $d\ge d_0$ and $L+d_0\ge\frac{3c^*T}{2}$ by~\eqref{defd0}, one has
$$\frac{3c^*T}{4}+x_0-\frac{L}{2}\le-\frac{3c^*T}{4}+x_0+d+\frac{L}{2}.$$
Call
\be\label{defI12}\left\{\begin{array}{l}
\displaystyle I_1:=\Big(\!\!-\frac{3c^*T}{4}+x_0-\frac{L}{2},\frac{3c^*T}{4}+x_0-\frac{L}{2}\Big),\vspace{3pt}\\
\displaystyle I_2:=\Big(\!\!-\frac{3c^*T}{4}+x_0+d+\frac{L}{2},\frac{3c^*T}{4}+x_0+d+\frac{L}{2}\Big).\end{array}\right.
\ee
These two intervals $I_1$ and $I_2$ are disjoint and the measure of $I_1\cup I_2$ is equal to $3c^*T$. The inequalities~\eqref{ineqint1}-\eqref{ineqint2} then imply that
\be\label{int12}
\int_{I_1\cup I_2}u(T,x)\,dx>\int_{-3c^*T/2}^{3c^*T/2}v(T,x)\,dx.
\ee
Recalling that $u^*(T,\cdot)$ is the Schwarz decreasing rearrangement of the bounded nonnegative function $u(T,\cdot)$, one concludes that
$$\int_{-3c^*T/2}^{3c^*T/2}u^*(T,x)\,dx\ge\int_{I_1\cup I_2}u(T,x)\,dx>\int_{-3c^*T/2}^{3c^*T/2}v(T,x)\,dx.$$
This gives~\eqref{eq:nomass} with $t:=T>0$ and $r:=3c^*T/2>0$.

\vskip 0.4cm
\noindent{}(ii) Let here $R$ be any positive real number. Consider $L>L^*\ge0$, $T>0$, and $d\ge d_0\ge0$ as in part~(i) of the proof, and let $\overline{u}$ and $\overline{v}$ be the solutions of~\eqref{eq:PDE} with respective initial data
$$\overline{u}_0:=\mathds{1}_{(-d/2-L,-d/2)\cup(d/2,d/2+L)}\ \hbox{ and }\ \overline{v}_0:=\overline{u}_0^*=\mathds{1}_{(-L,L)},$$
that is, $x_0:=-d/2$ in~\eqref{defu0}. Both solutions $\overline{u}$ and~$\overline{v}$ have maximal existence time~$+\infty$, range in $[0,1]$, and one knows by~\eqref{int12} that 
\be\label{intT12}
\int_{I_1\cup I_2}\overline{u}(T,x)\,dx>\int_{-3c^*T/2}^{3c^*T/2}\overline{v}(T,x)\,dx,
\ee
where the disjoint intervals $I_1$ and $I_2$ are given by~\eqref{defI12} with $x_0=-d/2$, that is,
$$I_1=\Big(\!\!-\frac{3c^*T}{4}-\frac{d}2-\frac{L}{2},\frac{3c^*T}{4}-\frac{d}2-\frac{L}{2}\Big),\ \ I_2=\Big(\!\!-\frac{3c^*T}{4}+\frac{d}2+\frac{L}{2},\frac{3c^*T}{4}+\frac{d}2+\frac{L}{2}\Big).$$

Now, for any real number $\rho$ such that
$$\rho>\rho_0:=\max\Big(\underbrace{\frac{3c^*T}{4}+\frac{d}{2}+\frac{L}{2}}_{\ge3c^*T/2}\,,\,\frac{d}{2}+L\Big)>0,$$
the following four sets are compactly contained in $(-\rho,\rho)$:
\begin{equation}\label{Subset}
(-d/2-L,-d/2)\cup(d/2,d/2+L), \ I_1\cup I_2, \
(-L,L), \ (-3c^*T/2,3c^*T/2).
\end{equation}
Let $u_\rho$ and $v_\rho$ be the solutions of~\eqref{eq:PDE-R} and~\eqref{eq:PDE2} in $I_\rho:=(-\rho,\rho)$ with respective initial data $\mathds{1}_{(-d/2-L,-d/2)\cup(d/2,d/2+L)}$ and $\mathds{1}_{(-L,L)}$, and with Dirichlet boundary conditions $u_\rho(t,\pm\rho)=v_\rho(t,\pm\rho)=0$ for $t>0$. From the comparison principle, both functions $u_\rho$ and $v_\rho$ are nonnegative, bounded by $1$, have maximal existence time~$+\infty$, and satisfy
\be\label{uvrho}
0\le u_\rho(t,x)\le\overline{u}(t,x)\le1\ \hbox{ and }\  0\le v_\rho(t,x)\le\overline{v}(t,x)\le1
\ee
for all $t>0$ and $x\in[-\rho,\rho]$. Observe that both $\overline{u}$ and $\overline{v}$ are even in $x$ (because so are $\overline{u}_0$ and $\overline{v}_0$), and denote
$$M_\rho:=\max_{t\in[0,T]}\max\big\{\overline{u}(t,\pm\rho),\overline{v}(t,\pm\rho)\big\},$$
which is a well defined nonnegative real number (and even positive from the strong maximum principle). Since $\overline{u}_0$ and $\overline{v}_0$ are compactly supported, there holds that $M_\rho\to0$ as $\rho\to+\infty$. Let $\Lambda$ be the Lipschitz norm of the function $f$ on the interval $[0,1]$. It then follows from~\eqref{uvrho} and the maximum principle that, for any $\rho>\rho_0$,
$$0\le \overline{u}(t,x)-u_\rho(t,x)\le e^{\Lambda t}M_\rho\le e^{\Lambda T}M_\rho\ \hbox { and }\ 0\le \overline{v}(t,x)-v_\rho(t,x)\le e^{\Lambda t}M_\rho\le e^{\Lambda T}M_\rho$$
for all $t\in(0,T]$ and $x\in[-\rho,\rho]$. Together with~\eqref{intT12} and $\lim_{\rho\to+\infty}M_\rho=0$, there exists then $\rho_1>\rho_0$ such that
\be\label{uvrho2}
\int_{I_1\cup I_2}u_\rho(T,x)\,dx>\int_{-3c^*T/2}^{3c^*T/2}v_\rho(T,x)\,dx,
\ee
for all $\rho\ge\rho_1$.

Finally, define
$$M_0:=\frac{\rho_1^2}{R^2}>0.$$
For any $M\ge M_0$, denote $\rho:=R\sqrt{M}\ge\rho_1$, and
$$u(t,x):=u_\rho\Big(\frac{\rho^2}{R^2}\,t,\frac{\rho}R\,x\Big),\ \ v(t,x):=v_\rho\Big(\frac{\rho^2}{R^2}\,t,\frac{\rho}R\,x\Big),\ \hbox{ for }(t,x)\in[0,+\infty)\times[-R,R].$$
The functions $u$ and $v$ take values in $[0,1]$ in $[0,+\infty)\times[-R,R]$, and they satisfy the equations~\eqref{eq:PDE-R} and~\eqref{eq:PDE2} in $(-R,R)$ with $Mf$ instead of $f$, with Dirichlet boundary conditions~\eqref{eq:dir} at $x=\pm R$, and with respective initial data
$$u_0:=\mathds{1}_{(-dR/(2\rho)-LR/\rho,-dR/(2\rho))\ \cup\ (dR/(2\rho),dR/(2\rho)+LR/\rho)} \ \hbox{ and }\ v_0=\mathds{1}_{(-LR/\rho,LR/\rho)}=u_0^*.$$
Since $(-\rho,\rho)$ compactly contains the sets listed in~\eqref{Subset},  the set
$$E:=\big(\!\!-dR/(2\rho)-LR/\rho,-dR/(2\rho)\big)\,\cup\,\big(dR/(2\rho),dR/(2\rho)+LR/\rho\big)$$
is compactly contained in $(-R,R)$, and so are $(-LR/\rho,LR/\rho)$, $R(I_1\cup I_2)/\rho$, and $(-3c^*TR/(2\rho),3c^*TR/(2\rho))$. Furthermore, calling
$$t:=\frac{TR^2}{\rho^2}>0\ \hbox{ and }\ r:=\frac{3c^*TR}{2\rho}>0,$$
the inequality~\eqref{uvrho2} entails that 
$$\int_{R(I_1\cup I_2)/\rho}u(t,x)\,dx>\int_{-r}^rv(t,x)\,dx.$$
Since the measure of $R(I_1\cup I_2)/\rho$ is equal to $2r$, this implies that
$$\int_{-r}^{r}u^*(t,x)\,dx>\int_{-r}^{r}v(t,x)\,dx.$$
The function $Mf$ and the solutions $u$ and $v$ then provide a counter-example to the mass concentration principle for the equations~\eqref{eq:PDE-R} and~\eqref{eq:PDE2} set in~$(-R,R)$ with Dirichlet boundary conditions~\eqref{eq:dir}. The proof of Theorem~\ref{thm:frag-fav2} is thereby complete.
\end{proof}

\begin{proof}[Proof of Theorem~\ref{thm:frag-fav3}] 
Let $g:\R_+\to\R$ be a $C^\infty$ function satisfying $g(0)=0$ together with the assumptions~\eqref{hyp:f1a} and~\eqref{hyp:f2} (with $f$ replaced by $g$) with any $p>3$ and $A>0$, or with $p\in(2,3]$ and $A>A^*$, where $A^*$ is given by Theorem~\ref{thm:limitL}-(ii). By Theorem~\ref{thm:limitL}, one has $L_\infty>0$ for that function $g$. Let $R_1$ and $R_2$ be such that $0<R_1<L_\infty<R_2$, and fix a positive real number $\alpha$ and an integer $n\ge2$ as in the statement of Theorem~\ref{thm:frag-fav} (with $f$ replaced by $g$). From the choice of $\alpha$ in~\eqref{def:alpha}, one knows that $\alpha>1$. Let
$$U_0:=\alpha\,\mathds{1}_{E_n(R_1,R_2)},$$
with $E_n(R_1,R_2)$ defined as in~\eqref{eq:En}, and notice that the Schwarz non-increasing rearrangement of $U_0$ is equal to 
$$U_0^*=\alpha\,\mathds{1}_{(-R_1/2,R_1/2)}.$$
Let $U$ and $V$ be the solutions of~\eqref{eq:PDE-R} and~\eqref{eq:PDE2} (with $R=+\infty$),  with initial data $U_0$ and $U_0^*$ respectively. They both have maximal existence time $T=+\infty$, and they are nonnegative and bounded in~$(0,+\infty)\times\R$ by $\alpha$ (since $g(s)<0$ for $s>1$). It follows from Theorem~\ref{thm:frag-fav} that $V(t,\cdot)\to0$ as~$t\to+\infty$ uniformly in $\R$, while $U(t,\cdot)\to1$ as $t\to+\infty$ locally uniformly in $\R$. Finally, define
$$f(s):=\frac{g(\alpha\,s)}{\alpha},\ \hbox{ for $s\in[0,1]$},$$
and
$$u(t,x):=\frac{U(t,x)}{\alpha},\ \ v(t,x):=\frac{V(t,x)}{\alpha},\ \hbox{ for $(t,x)\in\R_+\times\R$}.$$
The function $f$ is of class $C^\infty$ on $[0,1]$ with $f(0)=0$ and $f(1)=\frac{g(\alpha)}{\alpha}<0$. The functions $u$ and $v$ range in $[0,1]$, and they satisfy the equations~\eqref{eq:PDE-R} and~\eqref{eq:PDE2} in $\R$, with respective initial data
$$u(0,\cdot)=\mathds{1}_{E_n(R_1,R_2)},\ \ v(0,\cdot)=\mathds{1}_{(-R_1/2,R_1/2)}=u(0,\cdot)^*,$$
with $E:=E_n(R_1,R_2)$ bounded. Furthermore,
$$v(t,\cdot)\to0\ \hbox{ as~$t\to+\infty$ uniformly in $\R$,}$$
while
$$u(t,\cdot)\to\frac1\alpha>0\ \hbox{ as $t\to+\infty$ locally uniformly in $\R$}.$$
In particular, for any $r>0$, there holds
$$\int_{-r}^{r}u^*(t,x)\,dx>\int_{-r}^{r}v(t,x)\,dx$$
for every $t>0$ large enough. Therefore,~\eqref{eq:mass_comp} does not hold, for every $t>0$ large enough. The proof of Theorem~\ref{thm:frag-fav3} is thereby complete.
\end{proof}


\section*{Acknowledgements}

This work was supported by the ANR projects ReaCh, ANR-23-CE40-0023 and Maths-ArboV, ANR-24-EXMA-0004.

\bibliographystyle{abbrv}
{\footnotesize{
}}

\end{document}